\documentclass[a4paper]{article}
\usepackage[utf8]{inputenc}
\usepackage[english]{babel}
\usepackage{amsmath}
\usepackage{amssymb}
\usepackage{amsthm}
\usepackage{bbm}
\usepackage{graphicx,epstopdf}
\usepackage[ruled,linesnumbered]{algorithm2e}
\usepackage{tikz,tikz-3dplot}
\usetikzlibrary{calc}
\usetikzlibrary{decorations.pathmorphing}
\usepackage{pgfplots}
\usepackage{svg}
\usepackage{xcolor}
\usepackage{float}
\usepackage{booktabs,multirow}
\usepackage{amstext} 
\usepackage{array} 
\usepackage{hyperref}
\usepackage{enumerate}
\usepackage{titling}
\usepackage{a4wide}
\usepackage{mathtools}
\usepackage{color}
\usepackage[font={small}]{caption}

\usepackage{csquotes}

\newtheorem{theorem}{Theorem}
\theoremstyle{definition}

\newtheorem{proposition}[theorem]{Proposition}

\newtheorem{corollary}[theorem]{Corollary}
\theoremstyle{definition}
\newtheorem{example}[theorem]{Example}
\newtheorem{remark}[theorem]{Remark}

\newcommand{\eps}{\varepsilon}

\newcommand{\R}{\mathbb{R}}
\newcommand{\one}{\mathbbm{1}}

\DeclareMathOperator{\conv}{conv}

\DeclareMathOperator{\vertex}{vert}

\DeclareMathOperator{\inter}{int}

\newcommand*{\colonequals}{\mathrel{\vcenter{\baselineskip0.5ex%
      \lineskiplimit0pt\hbox{\scriptsize.}\hbox{\scriptsize.}}}=}
\renewcommand{\P}{\mathcal{P}}

\newcommand{\W}{\mathcal{W}}
\newcommand{\V}{\mathcal{V}}
\newcommand{\U}{\mathcal{U}}

\newcommand{\figurecoordinates}
{
	\coordinate (SW) at(-5,-3);
	\coordinate (NE) at(3.7,5.7);
	\coordinate (NULL) at (0,0);
	\coordinate (A1) at (-5,-1);
	\coordinate (A2) at (3,-1);
	\coordinate (AA1) at (-5,-2);
	\coordinate (AA2) at (3,-2);
	\coordinate (B1) at (-2,5);
	\coordinate (B2) at (-2,-3);
	\coordinate (BB1) at (-4,5);
	\coordinate (BB2) at (-4,-3);
	\coordinate (C1) at (1/2,-3);
	\coordinate (C2) at (-5,2.5);
	\coordinate (CC1) at (-2,-3);
	\coordinate (CC2) at (-5,0);
	\coordinate (D1) at (3,-2.5);
	\coordinate (D2) at (-4.5,5);
	\coordinate (DD1) at (3,-2);
	\coordinate (DD2) at (-4,5);
	\coordinate (E1) at (-7/4,-3);
	\coordinate (E2) at (1/4,5);
	\coordinate (EE1) at (-11/4,-3);
	\coordinate (EE2) at (-3/4,5);
	\coordinate (F1) at (-7/2,-3);
	\coordinate (F2) at (3,7/2);
	\coordinate (FF1) at (-4,-3);
	\coordinate (FF2) at (3,4);
	\coordinate (U1) at (-3/2,  -1);
	\coordinate (UU1) at (-3,  -2);
	\coordinate (V1) at (-1/2, -2);
	\coordinate (U2) at (-2  ,-1/2);
	\coordinate (UU2) at (-4  ,-1);
	\coordinate (V2) at (-17/8,-1/2);
	\coordinate (U3) at (3/2,-1);
	\coordinate (UU3) at (3,-2);
	\coordinate (V3) at (9/4,-3/2);
	\coordinate (U4) at (-2,5/2);
	\coordinate (UU4) at (-4,5);
	\coordinate (V4) at (-3,15/4);
	\coordinate (U5) at (-7/10,6/5);
	\coordinate (UU5) at (-7/5,12/5);
	\coordinate (V5) at (-21/20,9/5);
	\coordinate (U6) at (-5/4,-1);
	\coordinate (UU6) at (-5/2,-2);
	\coordinate (U7) at (0,1/2);
	\coordinate (UU7) at (0,1);
	\coordinate (U7) at (0,1/2);
	\coordinate (UU7) at (0,1);
	\coordinate (V7) at (0,3/4);
	\coordinate (U8) at (-7/6,-2/3);
	\coordinate (UU8) at (-7/3,-4/3);	
	\coordinate (V8) at (-7*4/17,-4*4/17);
}

\newcommand\SingleLine[8]{%
    \path(#5)--(#6)coordinate[at start](h1)coordinate[at end](h2);
    \draw[#7]($(h1)!#1!#2:(h2)$)-- node [auto=left] {#8} ($(h2)!#3!-#4:(h1)$); 
    }
\newcommand\DoubleLine[9]{%
    \path(#5)--(#6)coordinate[at start](h1)coordinate[at end](h2);
    \draw[#7]($(h1)!#1!#2:(h2)$)-- node [auto=left] {#8} ($(h2)!#3!-#4:(h1)$); 
    \draw[#9]($(h1)!#1!-#2:(h2)$)-- node [auto=right] {#8} ($(h2)!#3!#4:(h1)$);
    }

\begin{document}
\title{Approximate Vertex Enumeration}

\author{Andreas Löhne\textsuperscript{1}}

\footnotetext[1]{Friedrich Schiller University Jena, Department of
    Mathematics, 07737 Jena, Germany,
    andreas.loehne@uni-jena.de
}
\maketitle

\begin{abstract}
The problem to compute the vertices of a polytope given by affine inequalities is called vertex enumeration. 
The inverse problem, which is equivalent by polarity, is called the convex hull problem. 
We introduce `approximate vertex enumeration' as the problem to compute the vertices of a polytope which is close to the original polytope given by affine inequalities. 
In contrast to exact vertex enumerations, both polytopes are not required to be combinatorially equivalent.
 
Two algorithms for this problem are introduced. 
The first one is an approximate variant of Motzkin's double description method. 
Only under certain strong conditions, which are not acceptable for practical reasons, we were able to prove correctness of this method for polytopes of arbitrary dimension. 
The second method, called shortcut algorithm, is based on constructing a plane graph and is restricted to polytopes of dimension $2$ and $3$. 
We prove correctness of the shortcut algorithm. 
As a consequence, we also obtain correctness of the approximate double description method, only for dimension $2$ and $3$ but without any restricting conditions as still required for higher dimensions. 
We show that for dimension $2$ and $3$ both algorithm remain correct if imprecise arithmetic is used and the computational error caused by imprecision is not too high. 
Both algorithms were implemented. 
The numerical examples motivate the approximate vertex enumeration problem by showing that the approximate problem is often easier to solve than the exact vertex enumeration problem.

It remains open whether or not the approximate double description method (without any restricting condition) is correct for polytopes of dimension $4$ and higher. 

\medskip

 \end{abstract}

\section{Problem formulation and introduction}

Let $P$ be an H-polytope (i.e.\ a bounded polyhedron represented by affine inequalities) with zero in its interior. 
Setting $\one=(1,\dots,1)^T$, $P$ can be expressed by a matrix $A \in \R^{m\times d}$ as
\begin{equation*}
	P = \{x \in \R^d \mid A x \leq \one\}.
\end{equation*}
For some  tolerance $\eps \geq 0$ we define
\begin{equation*}
	(1+\eps) P = \{ x \in \R^d\mid A x \leq (1+\eps) \one\}.
\end{equation*} 

The goal is to construct iteratively an {\em ($\eps$-)approximate V-representation}, that is, a finite set $\V = \{v_1,\dots,v_k\} \subseteq \R^d$ such that the convex hull $Q \colonequals \conv \V$ of $\V$ satisfies
\begin{equation}\label{eq_approx}
	P \subseteq Q \subseteq (1+\eps) P.
\end{equation}
For $\varepsilon=0$ we obtain a V-represention of $P$. 
Of course, any V-representation of $P$ or $(1+\eps)P$ is also an approximate V-representation of $P$. 
Allowing more approximate V-representations by enlarging the tolerance $\eps \geq 0$ can reduce the computational time and the complexity of the result (less vertices in $\V$) as will be demonstrated by numerical examples. 
In the present approach, the combinatorial structure (more precisely, the face lattice) of the V-polytope $Q$ can be (and usually is) different from the combinatorial structure of the given H-polytope $P$. 
Thus our problem setting does not aim to get any combinatorial information (like a facet-vertex incidence list) of $P$. 
Instead we only obtain a V-polytope $Q$ that approximates the H-polytope $P$ by a prescribed tolerance. 
Potential applications of this problem setting and the presented methods can be seen in the field of approximation of convex sets, see e.g.\ \cite{Bronstein08}, which includes solution methods for vector and set optimization problems, see e.g.\ \cite{5pp}. 

Vertex enumeration algorithms are usually not ``stable'' with respect to imprecise computations, for instance, caused by floating point arithmetic. 
To get a flavor of possible computational problems in geometry caused by inexact arithmetic, even in the plane, the reader is referred to \cite{Kettner04}.
Also in Figure \ref{fig_fail} below, we see that imprecision can lead to results which are not even approximations of the correct ones.
We aim to make aware with this article that floating point implementations of the vertex enumeration or convex hull problem do not evidently compute correct results. Even more important is to keep in mind that for such methods, which are frequently used in practice, e.g. \cite{qhull_software, bt, bt-paper}, there is not even any evidence that they compute approximations of the correct results (except for $2$-dimensional problems, see Section \ref{sec_lit}).    

For the latter issue there is a significant difference between polytopes of dimension up to $3$ and those of higher dimension. While we are able to provide (practically relevant) correct methods for the approximate vertex enumeration problem up to dimension $3$, this question remains open for higher dimensions. Since our proof technique is based on the planarity of the vertex-edge graph of a polytope, it cannot be applied to polytopes of dimension $d \geq 4$. The difference between 3- and 4-polytopes (i.e.\ polytopes of dimension $3$ and $4$) can also be observed from a theoretical point of view.
For dimension larger than 3, i.e.\ beyond the sphere of validity of Steinitz' theorem, arbitrarily small local changes of the data can cause global changes of the combinatorics, see e.g.\ the examples in \cite[Section 4]{Ziegler95}. 
The double description method, see e.g. \cite{ddm}, constructs iteratively a sequence of polytopes by adding inequalities. 
In each iteration the facet-vertex incidence information is used. 
The incidence list is updated only locally in each step after adding an inequality. 
Thus global changes in the combinatorics can result in invalid incidence information. 
Moreover, the universality theorem of Richter-Gebert \cite{Richter-Gebert96} states that the realization space (i.e.\ the space of all polytopes being combinatorially equivalent) of a 4-polytope can be ``arbitrarily bad'' \cite{RicZie95}. 
This has several consequences, in particular, all algebraic numbers are needed to realize 4-polytopes \cite{RicZie95}. 
In contrast to this, due to Steinitz' theorem, see e.g. \cite{Ziegler95}, 3-polytopes are realizable by integral coordinates. These fundamental differences between 3- and 4-polytopes, the mentioned lack of evidence, and the fact that there are only limited options of testing a result for plausibility show that floating point implementations of geometric algorithms for $d\geq 4$ should be treated with caution.

This article is organized as follows. 
In Section \ref{basic_cut} we discuss an extension of the basic cutting scheme of the exact double description method. 
The main idea is that the classical cutting hyperplane is replaced by the space between two parallel hyperplanes. 
We show that certain ingeniuous extensions of the double description method fail. 
Only under strong and impracticable assumptions, these ideas lead to a correct algorithm. 
We show by examples that the assumptions cannot be omitted in the present approach. 
We close this section with the guess that a straightforward extension of the correctness results for exact vertex enumeration methods to approximate methods is not possible.
In Section \ref{sec_addm} we formulate the approximate double description method. 
Section \ref{sec_23} is devoted to polytopes of dimension $2$ and $3$.
We introduce the shortcut algorithm and prove its correctness. 
As a consequence we obtain also correctness of the approximate double description method, without any further assumptions but for dimension $2$ and $3$ only.
In Section \ref{sec_imprec} we show that the use of imprecise arithmetic maintains the correctness results from the previous section if the imprecision is not too high. Some numerical results are presented in Section \ref{sec_num}.
They are used to compare the two methods and to demonstrate some benefits of the approximate vertex enumeration. We close with some conclusions and open questions.

%
%

\section{Discussion of related work} \label{sec_lit}

Concepts like {\em robustness} and {\em stability}, see e.g. \cite{Fortune89,HofHopKar88,Sugihara94, Sugihara00} were introduced and studied in order to ``verify'' (geometrical) algorithms performed with imprecise arithmetic. These concepts can be seen as generalizations of the classical notion of correctness of an algorithm. If a geometric algorithm is performed with imprecise arithmetic, correct results (in the sense of exact results for the original problem) cannot be expected. This is one reason why the notion of correctness needs to be generalized. The approach in this paper is a different one. We do not need a generalized concept of correctness. Instead we adapt the problem setting:
	\begin{itemize}
		\item First we define an approximate problem, where the goal is to solve the problem only approximately.
		\item Secondly, we look for correct algorithms (in the classical sense) solving this (weaker) problem. The arithmetic used here (such as exact or floating point) is seen as a convention which is part of the algorithm. Then the only goal is to show correctness of the method, that is: the algorithm is well-defined for any valid input and its output is an exact solution of the approximate problem. 
	\end{itemize}
Despite of this different approach, the goal to ``verify'' algorithms performed with imprecise arithmetic is the same. Moreover, our approach is related to {\em robustness} and {\em stability} in the following sense: For geometric algorithms in $\eps$-arithmetic Fortune \cite{Fortune89} generalizes the concept of {\em correctness} as follows: First, such an algorithm is required to be correct (in the classical sense) if implemented in exact arithmetic. Secondly, it is called {\em robust} if it always produces a result that is the correct one for some perturbation of its input; it is called {\em stable} if, in addition, the perturbation is small. Assume that an H-polytope $P$ is given and consider $\V=\{v_1,\dots,v_k\} \subseteq \R^d$ satisfying $P=\conv \V$ as the ``correct'' result in Fortune's definition. Solving the approximate vertex enumeration problem for $P$ and some prescribed tolerance $\eps>0$ using floating point arithmetic under the conditions discussed in Section \ref{sec_imprec} we obtain a set $\W=\{w_1,\dots,w_\ell\}$ such that $P \subseteq \conv \W \subseteq (1+\eps)P$. An H-representation $Q$ of $\conv \W$ exists and can be considered as a perturbed input in Fortune's definition. Then $\W$ is the ``correct'' result for this perturbed input and thus our method is robust. The perturbation of the input is small in the sense that $P \subseteq Q \subseteq (1+\eps)P$, which means that our method is also stable.  

There is a vast amount of literature related to issues with imprecision of geometric computations, see e.g. \cite[Section 4]{Schirra00} for an overview.
Most of the literature is formulated in terms of {\em convex hull computation}, a problem being equivalent to vertex enumeration by polarity. 
The major part of the work on convex hull algorithms suitable for floating point arithmetic is done for two dimensions, for a selection of literature see e.g. \cite[Section 4.6]{Schirra00}. 
Subsequently we focus on literature on problems with three and more dimensions. To the best of our knowledge the methods introduced in this article differ from those in the literature.

Sugihara \cite{Sugihara94} presents a version of the gift-wrapping algorithm for 3-polytopes, which is robust and topologically consistent but not stable (where the terms {\em robust} and {\em stable} are used in a slightly different manner).
In \cite{HofHopKar88} ``robust'' computations of intersections of polygons are studied. The case of three dimensions (which is reduced to a sequence of polygon intersections) is also discussed, but there is no proof of correctness of the algorithm. Later, in \cite{HopKah92}, {\em correctness for geometric algorithms} has been defined, a concept that takes into account (imprecise) representations of geometric objects. The authors show correctness (in the sense of their definition) for intersecting a 3-polytope with a half-space. 
But iteration, as required for convex hull computations, does not necessarily lead to a correct algorithm as discussed in \cite[Section 4.5]{HopKah92}. 
Moreover, the results in \cite{HopKah92} require some sufficiently small approximation error which is not known a priori.

Even though floating point implementations of vertex enumeration (or convex hull) algorithms are frequently used in practice to ``solve''  problems in arbitrary dimensions (see e.g. \cite{BarDobHuh96, qhull_software, LoeWei16, bensolve, AviBreSei97}), it seems there is no stability result for dimension $d\geq 3$ and no robustness result for dimension $d \geq 4$ so far.


\section{A basic cutting scheme}\label{basic_cut}

The (exact) double description method is an iterative scheme where the inequalities of the H-representation are added one by one. 
After an inequality has been added, the V-representation of the intermediate result is updated. 
In this section we discuss an extension to an approximate variant of the method. To this end we describe a typical step of this method. Let $P$ be defined by the inequalities that have been added so far by the algorithm. Assume that an $\eps$-approximate V-representation $\V$ of $P$ already has been computed. Then we add in such a step a new inequality $h^T x \leq 1$. Denoting the corresponding half-space by $H_\leq \colonequals\{x \mid h^T x \leq 1\}$, we aim to compute an $\eps$-approximate V-representation $\V'$ of $P'\colonequals P\cap H_\leq$.  

For $h\in \R^d\setminus\{0\}$ we set 
\begin{equation*}
		H_+ \colonequals \{x \mid h^T x > 1 + \eps\},\quad	H_- \colonequals \{x \mid h^T x < 1\},\quad
		H_0\; \colonequals \{x \mid 1 \leq h^T x \leq 1+\eps\}.
\end{equation*} 
For $I \subseteq [m]$ we denote by $A_I$ the submatrix of $A$ which consists of the rows of $A$ with indices $I$. If $I=\{i\}$, we also write $A_i$ instead of $A_{\{i\}}$. 
For $\prec \in \{>, \geq, =,\neq, <, \leq\}$ and $u \in \R^d$, we write
\begin{equation*}
	J_\prec(u) = \{i \in [m] \mid A_i u \prec 1\}.
\end{equation*}
In Algorithm~\ref{alg_basic_cut} we give the pseudocode for our cutting scheme. 

For $\eps=0$, Algorithm~\ref{alg_basic_cut} is similar to a typical iteration step of the (exact) double description method:
If $v_1,v_2$ are vertices of $P$ that are endpoints of an edge of $P$ then $v_1$ and $v_2$ have at least $d-1$ common incident inequalities, i.e.\ $|J_=(v_1) \cap J_=(v_2)| \geq d-1$. The condition in line \ref{line_cond_edge} of Algorithm~\ref{alg_basic_cut} generalizes this necessary condition.

\begin{algorithm}[hpt]
\DontPrintSemicolon
\SetKwFor{loop}{loop}{}{end}
\KwIn{polytope $P$, $\eps>0$,\; half-space $H_\leq \colonequals \{x \mid h^T x \leq 1\}$,\; $\eps$-approximate V-representation $\V$ of $P$}
\KwOut{$\eps$-approximate V-representation $\V'$ of $P' \colonequals P \cap H_\leq$}
\Begin{
$\V_+ \leftarrow \V \cap H_+$\;
$\V_- \leftarrow \V \cap H_-$\;
$\V_0 \leftarrow \V \cap H_0$\;
\ForEach{$(v_1,v_2) \in \V_- \times \V_+$}
{ 
	\If{$|J_\geq(v_1) \cap J_\geq(v_2)| \geq d-1$\label{line_cond_edge}}
	{
		choose some $v$ in the line segment $v_1 v_2$ that belongs to $H_0$\label{alg1_line7}\;
		$\V \leftarrow \V \cup \{v\}$\label{alg1_line8}\;
	}
}
$\V' \leftarrow \V\setminus \V_+$\; \label{cutoff}
}
\caption{\label{alg_basic_cut}%
  Basic Cut}
\end{algorithm}

Figure \ref{fig_fail} shows that Algorithm~\ref{alg_basic_cut} is not correct, not even in the plane, as the output is not an $\eps$-approximate V-representation of $P'$. Therefore we try for a modification, which requires some concepts and preliminary results. 

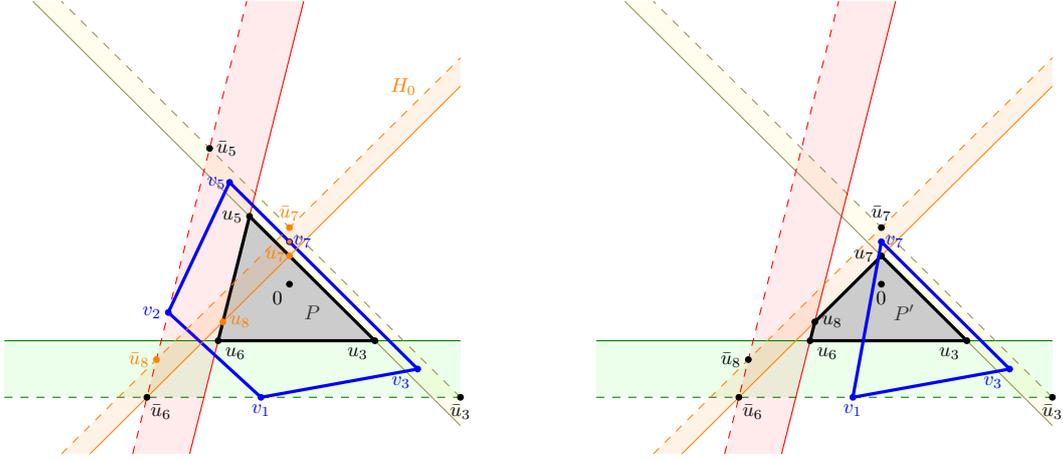
\begin{figure}[hpt]
	\def\scalefactor{.75}	
	\def\spacebetweenpictures{1cm}
\begin{center}
\begin{tikzpicture}[scale=\scalefactor,every node/.style={scale=\scalefactor}]
\figurecoordinates
\draw[white,fill=white] (NE) circle (0.01);
\draw[white,fill=white] (SW) circle (0.01);

\draw [opacity=0, fill opacity=0.1,fill=green!80!white] (A1) -- (A2) -- (AA2) -- (AA1) -- cycle;
\draw [opacity=0, fill opacity=0.1,fill=yellow!80!white] (D1) -- (D2) -- (DD2) -- (DD1) -- cycle;
\draw [opacity=0, fill opacity=0.1,fill=red!80!white] (E1) -- (E2) -- (EE2) -- (EE1) -- cycle;

\draw[green!50!black] (A1) -- (A2);
\draw[green!50!black,dashed] (AA1) -- (AA2);
\draw[yellow!50!black] (D1) -- (D2);
\draw[yellow!50!black,dashed] (DD1) -- (DD2);
\draw[red] (E1) -- (E2);
\draw[red,dashed] (EE1) -- (EE2);

\draw[blue,very thick] (V1) -- (V3) -- (V5) -- (V2) --  cycle;
\draw[fill=gray!40!white,very thick] (U5) -- (U6) -- (U3) --  cycle;

\draw[red!50!yellow] (F1) -- (F2);
\draw[red!50!yellow,dashed] (FF1) -- (FF2);
\draw [opacity=0, fill opacity=0.1,fill=red!50!yellow] (F1) -- (F2) -- (FF2) -- (FF1) -- cycle;

\draw[black!80!white] (.4,-.5) node {$P$};
\draw[red!50!yellow] (2,7/2) node {$H_0$};

\draw[fill=black] (NULL) circle (0.05) node[below left] {$0$};

\draw[fill=black] (U3) circle (0.05) node[below left] {$u_3$};
\draw[fill=black] (U5) circle (0.05) node[left] {$u_5$};
\draw[fill=black] (U6) circle (0.05) node[below right] {$u_6$};
\draw[red!50!yellow,fill=red!50!yellow] (U7) circle (0.05) node[left] {$u_7\!$};
\draw[red!50!yellow,fill=red!50!yellow] (U8) circle (0.05) node[right] {$u_8\!$};

\draw[blue, fill=blue] (V1) circle (0.05) node[below] {$v_1$};
\draw[blue, fill=blue] (V2) circle (0.05) node[left] {$v_2$};
\draw[blue, fill=blue] (V3) circle (0.05) node[below left] {$v_3$};
\draw[blue, fill=blue] (V5) circle (0.05) node[left] {$v_5\!$};
\draw[blue,fill=red!50!yellow] (V7) circle (0.05) node[right] {$\!v_7$};

\draw[fill=black] (UU3) circle (0.05) node[below] {$\bar u_3$};
\draw[fill=black] (UU5) circle (0.05) node[right] {$\bar u_5$};
\draw[fill=black] (UU6) circle (0.05) node[below right] {$\!\bar u_6$};
\draw[red!50!yellow,fill=red!50!yellow] (UU7) circle (0.05) node[above] {$\bar u_7$};
\draw[red!50!yellow,fill=red!50!yellow] (UU8) circle (0.05) node[left] {$\bar u_8$};
\end{tikzpicture}
\hspace{\spacebetweenpictures}
\begin{tikzpicture}[scale=\scalefactor,every node/.style={scale=\scalefactor}]

\draw[white,fill=white] (NE) circle (0.01);
\draw[white,fill=white] (SW) circle (0.01);

\draw [opacity=0, fill opacity=0.1,fill=green!80!white] (A1) -- (A2) -- (AA2) -- (AA1) -- cycle;
\draw [opacity=0, fill opacity=0.1,fill=yellow!80!white] (D1) -- (D2) -- (DD2) -- (DD1) -- cycle;
\draw [opacity=0, fill opacity=0.1,fill=red!80!white] (E1) -- (E2) -- (EE2) -- (EE1) -- cycle;
\draw [opacity=0, fill opacity=0.1,fill=red!50!yellow] (F1) -- (F2) -- (FF2) -- (FF1) -- cycle;

\draw[green!50!black] (A1) -- (A2);
\draw[green!50!black,dashed] (AA1) -- (AA2);
\draw[yellow!50!black] (D1) -- (D2);
\draw[yellow!50!black,dashed] (DD1) -- (DD2);
\draw[red] (E1) -- (E2);
\draw[red,dashed] (EE1) -- (EE2);
\draw[red!50!yellow] (F1) -- (F2);
\draw[red!50!yellow,dashed] (FF1) -- (FF2);

\draw[fill=gray!40!white,very thick] (U6) -- (U3) -- (U7) -- (U8) --  cycle;

\draw[black!80!white] (.4,-.5) node {$P'$};

\draw[fill=black] (NULL) circle (0.05) node[below] {$0$};

\draw[fill=black] (U3) circle (0.05) node[below left] {$u_3$};
\draw[fill=black] (U6) circle (0.05) node[below right] {$u_6$};
\draw[fill=black] (U7) circle (0.05) node[left] {$u_7$};
\draw[fill=black] (U8) circle (0.05) node[right] {$u_8\!$};

\draw[blue,very thick] (V1) -- (V3) -- (V7) --  cycle;
\draw[blue, fill=blue] (V1) circle (0.05) node[below] {$v_1$};
\draw[blue, fill=blue] (V3) circle (0.05) node[below left] {$v_3$};
\draw[blue, fill=blue] (V7) circle (0.05) node[right] {$\!v_7$};

\draw[fill=black] (UU3) circle (0.05) node[below] {$\bar u_3$};
\draw[fill=black] (UU6) circle (0.05) node[below right] {$\!\bar u_6$};
\draw[fill=black] (UU7) circle (0.05) node[above] {$\bar u_7$};
\draw[fill=black] (UU8) circle (0.05) node[left] {$\bar u_8$};
\end{tikzpicture}
\end{center}
\caption{Algorithm~\ref{alg_basic_cut} can fail, even in the plane.
	Left: $P= \conv\{u_3,u_5,u_6\}$, $(1+\eps)P= \conv\{\bar u_3,\bar u_5,\bar u_6\}$ ($\eps = 1$) and a new inequality, illustrated by $H_0$. 
	The set $\V=\{v_1,v_3,v_5,v_2\}$ provides an $\eps$-approximate V-representation of $P$. 
	Right: Result of Algorithm~\ref{alg_basic_cut}. It computes $\V'=\{v_1,v_3,v_7\}$, which is not an $\eps$-approximate V-representation of $P'$ since $P' = \conv\{u_3,u_7,u_8,u_6\} \not\subseteq \conv \V'$.}\label{fig_fail}
\end{figure}

A vertex $u$ of $P$ is said to be {\em covered} by a point $v \in \R^d$ if
\begin{equation*}
	u \in \conv \left(\vertex P \setminus \{u\} \cup \{v\}\right).
\end{equation*}
This means that, if we replace a vertex $u$ of $P$ in the set $\vertex P$ of vertices of $P$ by a point $v$, then this new set V-represents a superset of $P$. 

\begin{proposition}  \label{prop_cover2} 
For $u \in \vertex P$ and $v \in \R^d$, the following statements are equivalent:
	\begin{enumerate}[(i)]
		\item $v$ covers $u$,
		\item $J_=(u) \subseteq J_\geq(v)$.
	\end{enumerate}
\end{proposition}

\begin{proof} (i) $\Rightarrow$ (ii). Assume $v$ covers $u$ but there is $j \in J_=(u)$ with $A_j v < 1$. We have $u=\lambda v + (1-\lambda) z$ for some $z \in \conv(\vertex P \setminus\{u\})$ and $0 < \lambda \leq 1$, where $\lambda \neq 0$ holds as $u$ is a vertex of $P$. Thus we obtain the contradiction $1 = A_{i} u = \lambda A_{i} v + (1-\lambda) A_{i} z < 1$.
	
	(ii) $\Rightarrow$ (i). Let $u \in \vertex P$. For $u=v$ the statement is obvious, thus assume $u\neq v$. Set $u_\gamma = u + \gamma (u-v)$. For $\gamma > 0$ we have $J_=(u) \subseteq J_\leq(u_\gamma)$. If $\gamma > 0$ is sufficiently small then 
$J_<(u) \subseteq J_\leq (u_\gamma)$. Thus $[m] = J_=(u) \cup J_<(u) \subseteq J_\leq(u_\gamma)$ and hence $u_\gamma \in P$ for sufficiently small $\gamma > 0$. Let $\mu = \sup \{\gamma \geq 0 \mid u_\gamma \in P\}$. Since $P$ is compact and $u \neq v$, $\mu>0$ is finite and $u_\mu \in P$. Since $u$ is a convex combination of $v$ and $u_\mu$, it remains to show that $u_\mu \in  \conv(\vertex P \setminus \{u\})$. Assuming the contrary we obtain $u_\mu = \lambda u + (1-\lambda) w$ for some $\lambda \in (0, 1)$ and some $w \in P$. Then $u_{\frac{\mu}{1-\lambda}} = w \in P$ contradicts the maximality of $\mu$. 
\end{proof}

The idea is now to compute a set of points such that all vertices of $P$ are covered. If so, edges of $P$ can be related to the condition in line \ref{line_cond_edge} of Algorithm~\ref{alg_basic_cut}: Let $u_1,u_2$ be vertices of $P$ that are endpoints of an edge of $P$. Then we have $|J_=(u_1) \cap J_=(u_2)| \geq d-1$. If $v_1$ covers $u_1$ and $v_2$ covers $u_2$, Proposition \ref{prop_cover2} yields that the condition in line \ref{line_cond_edge} of Algorithm~\ref{alg_basic_cut} is necessary 
for $u_1,u_2$ being endpoints of an edge of $P$.

Let $\eps \geq 0$ be a given tolerance. A finite set $\V \subseteq (1+\eps)P$ is called a {\em strong $\eps$-approximate V-representation} of $P$ if every vertex of $P$ is covered by some element of $\V$. The following proposition tells us that every strong $\eps$-approximate V-representation of $P$ is also an $\eps$-approximate V-representation of $P$ (replace $\conv \V$ by $\V$ in the second condition).

\begin{proposition}\label{prop_cover}
The following statements are equivalent:
		\begin{enumerate}[(i)]
			\item $\V$ is an $\eps$-approximate V-representation of $P$,
			\item $\V \subseteq (1+\eps)P$ and every vertex of $P$ is covered by some element of $\conv \V$.
		\end{enumerate}
\end{proposition}

\begin{proof} (i) $\implies$ (ii). This is obvious since every vertex of $P \subseteq Q$ covers itself.
	
(ii) $\implies$ (i). We need to show that $P \subseteq \conv \V \subseteq (1+\varepsilon)\P$. The second inclusion is a direct consequence of $\V \subseteq (1+\eps)P$. To show the first inclusion we denote by $\{u_i \mid i \in [k]\}$ the set of vertices of $P$. Let $v_i \in \conv\V$ be a point that covers $u_i$ (this allows $v_i=v_j$ for $u_i\neq u_j$). For an index set $I \subseteq [k]$ we write $\U_I=\{u_i \mid i \in I\}$ and $\V_I=\{v_i \mid i \in I\}$.
We show by induction that, for all $\ell \in [k]$,
\begin{equation}\label{eq_ind}
	\forall I \subseteq [k] \text{ with } |I|=\ell:\quad \U_I \subseteq \conv (\V_I \cup \U_{[k]\setminus I}). 
\end{equation}
For $\ell=k$ this leads to $\U_{[k]} \subseteq \conv \V_{[k]}$, which proves our claim.

For $\ell=1$, \eqref{eq_ind} holds because, for all $i \in [k]$, $u_i$ is covered by $v_i$. Assume that \eqref{eq_ind} holds for some $\ell = n < k$. Let $I \subseteq [k]$ with $|I|=n+1$. Without loss of generality let $I=[n+1]$. We show that $u_1 \in \conv (\V_I \cup \U_{[k]\setminus I})$. Let $I_2=I\setminus \{2\}$ and $I_1=I\setminus \{1\}$. Since $|I_1| = |I_2| = n$, we have
$u_1 \in \conv (\V_{I_2} \cup \U_{[k]\setminus {I_2}})$ and $u_2 \in \conv (\V_{I_1} \cup \U_{[k]\setminus {I_1}})$. Thus there are $\lambda_i \geq 0$ ($i\in [k]$) with $\sum_{i=1}^k \lambda_i = 1$ and $\mu_i \geq 0$ ($i\in [k]$) with $\sum_{i=1}^k \mu_i = 1$ such that
\begin{equation}\label{eq_ind1}
	u_1 = \lambda_1 v_1 + \lambda_2 u_2 + \lambda_3 v_3 + \dots + \lambda_{n+1} v_{n+1} + \lambda_{n+2} u_{n+2} + \dots + \lambda_k u_k,
\end{equation} 
\begin{equation}\label{eq_ind2}
u_2 = \mu_1 u_1 + \mu_2 v_2 + \mu_3 v_3 + \dots + \mu_{n+1} v_{n+1} + \mu_{n+2} u_{n+2} + \dots + \mu_k u_k.\end{equation} 
We substitute $u_2$ in \eqref{eq_ind1} by the right hand side of \eqref{eq_ind2} and resolve by $u_1$ (since $u_1 \neq u_2$ we have $1-\lambda_2 \mu_1 \neq 0$). We obtain that $u_1$ is a convex combination of $\{v_1,\dots,v_{n+1},u_{n+2},\dots,u_k\}$, i.e.\ $u_1 \in  \conv (\V_I \cup \U_{[k]\setminus I})$. Likewise we get $u_i \in  \conv (\V_I \cup \U_{[k]\setminus I})$ for all $i \in I$. Hence \eqref{eq_ind} holds for $\ell = n+1$, which completes the proof.
\end{proof}

A point $v \in \R^d$ is said to be {\em $\eps$-incident} to (an inequality indexed by) $i \in [m]$ iff $v \in (1+\eps)P$ and $i \in J_\geq(v)$. By Proposition~\ref{prop_cover2}, every inequality $i \in [m]$ that is incident with a vertex $u$ of $P$ (i.e.\ $i \in J_=(u)$) is $\eps$-incident with a point $v \in (1+\eps)P$ that covers $u$. The converse is not true. For instance, in $\R^2$ one can easily construct an example where $v \in (1+\eps)P$ covers two vertices $u_1 \neq u_2$ such that $J_\geq(v) = \{1,2,3\}$ and $J_=(u_1) = \{1,2\}$, $J_=(u_2)= \{2,3\}$. Inequality $3$ is $\varepsilon$-incident to $v$, which covers $u_1$, but not incident to $u_1$.

If the input of Algorithm~\ref{alg_basic_cut} is required to be a strong $\eps$-approximate V-representation, the result $\V'$ can be an $\eps$-approximate V-representation, but not necessarily a strong one, see Figure~\ref{fig1a}. Therefore, repeated application of the algorithm could also fail in this case. But, as we see in Figure~\ref{fig1b}, this is not the case in our example. We will show in Section \ref{sec_23} (based on substantially other arguments) that for dimension $d\leq 3$ repeated application of the algorithm always works.

\smallskip
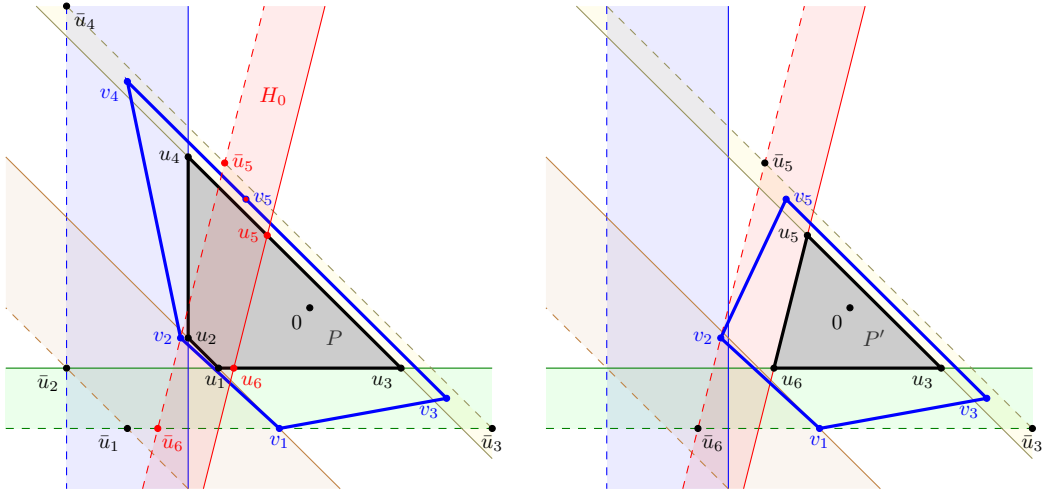
\begin{figure}[hpt]
	\def\scalefactor{.75}	
	\def\spacebetweenpictures{1cm}
\begin{center}
\begin{tikzpicture}[scale=\scalefactor,every node/.style={scale=\scalefactor}]
\figurecoordinates

\draw[white,fill=white] (NE) circle (0.01);
\draw[white,fill=white] (SW) circle (0.01);

\draw [opacity=0, fill opacity=0.1,fill=green!80!white] (A1) -- (A2) -- (AA2) -- (AA1) -- cycle;
\draw [opacity=0, fill opacity=0.1,fill=blue!80!white] (B1) -- (B2) -- (BB2) -- (BB1) -- cycle;
\draw [opacity=0, fill opacity=0.1,fill=brown!80!white] (C1) -- (C2) -- (CC2) -- (CC1) -- cycle;
\draw [opacity=0, fill opacity=0.1,fill=yellow!80!white] (D1) -- (D2) -- (DD2) -- (DD1) -- cycle;

\draw[green!50!black] (A1) -- (A2);
\draw[green!50!black,dashed] (AA1) -- (AA2);
\draw[blue] (B1) -- (B2);
\draw[blue,dashed] (BB1) -- (BB2);
\draw[brown] (C1) -- (C2);
\draw[brown,dashed] (CC1) -- (CC2);
\draw[yellow!50!black] (D1) -- (D2);
\draw[yellow!50!black,dashed] (DD1) -- (DD2);

\draw[fill=gray!40!white,very thick] (U1) -- (U2) -- (U4) -- (U3) -- cycle;
\draw[blue,very thick] (V1) -- (V3) -- (V4) -- (V2) --  cycle;

\draw [opacity=0, fill opacity=0.1,fill=red!80!white] (E1) -- (E2) -- (EE2) -- (EE1) -- cycle;
\draw[red] (E1) -- (E2);
\draw[red,dashed] (EE1) -- (EE2);

\draw[black!80!white] (.4,-.5) node {$P$};
\draw[red] (-.6,3.5) node {$H_0$};

\draw[fill=black] (NULL) circle (0.05) node[below left] {$0$};

\draw[fill=black] (U1) circle (0.05) node[below] {$u_1\;$};
\draw[fill=black] (U2) circle (0.05) node[right] {$u_2\;$};
\draw[fill=black] (U3) circle (0.05) node[below left] {$u_3$};
\draw[fill=black] (U4) circle (0.05) node[left] {$u_4$};
\draw[red,fill=red] (U5) circle (0.05) node[left] {$u_5$};
\draw[red,fill=red] (U6) circle (0.05) node[below right] {$u_6$};

\draw[blue, fill=blue] (V1) circle (0.05) node[below] {$v_1$};
\draw[blue, fill=blue] (V2) circle (0.05) node[left] {$v_2$};
\draw[blue, fill=blue] (V3) circle (0.05) node[below left] {$v_3$};
\draw[blue, fill=blue] (V4) circle (0.05) node[below left] {$v_4$};
\draw[blue, fill=red] (V5) circle (0.05) node[right] {$v_5\!$};

\draw[fill=black] (UU1) circle (0.05) node[below left] {$\bar u_1$};
\draw[fill=black] (UU2) circle (0.05) node[below left] {$\bar u_2$};
\draw[fill=black] (UU3) circle (0.05) node[below] {$\bar u_3$};
\draw[fill=black] (UU4) circle (0.05) node[below right] {$\bar u_4$};
\draw[red,fill=red] (UU5) circle (0.05) node[right] {$\bar u_5$};
\draw[red,fill=red] (UU6) circle (0.05) node[below right] {$\!\bar u_6$};
\end{tikzpicture}
\begin{tikzpicture}[scale=\scalefactor,every node/.style={scale=\scalefactor}]
\draw[white,fill=white] (NE) circle (0.01);
\draw[white,fill=white] (SW) circle (0.01);

\draw [opacity=0, fill opacity=0.1,fill=green!80!white] (A1) -- (A2) -- (AA2) -- (AA1) -- cycle;
\draw [opacity=0, fill opacity=0.1,fill=blue!80!white] (B1) -- (B2) -- (BB2) -- (BB1) -- cycle;
\draw [opacity=0, fill opacity=0.1,fill=brown!80!white] (C1) -- (C2) -- (CC2) -- (CC1) -- cycle;
\draw [opacity=0, fill opacity=0.1,fill=yellow!80!white] (D1) -- (D2) -- (DD2) -- (DD1) -- cycle;
\draw [opacity=0, fill opacity=0.1,fill=red!80!white] (E1) -- (E2) -- (EE2) -- (EE1) -- cycle;

\draw[green!50!black] (A1) -- (A2);
\draw[green!50!black,dashed] (AA1) -- (AA2);
\draw[blue] (B1) -- (B2);
\draw[blue,dashed] (BB1) -- (BB2);
\draw[brown] (C1) -- (C2);
\draw[brown,dashed] (CC1) -- (CC2);
\draw[yellow!50!black] (D1) -- (D2);
\draw[yellow!50!black,dashed] (DD1) -- (DD2);
\draw[red] (E1) -- (E2);
\draw[red,dashed] (EE1) -- (EE2);

\draw[blue, very thick] (V1) -- (V3) -- (V5) -- (V2) --  cycle;
\draw[fill=gray!40!white, very thick] (U5) -- (U6) -- (U3) --  cycle;

\draw[black!80!white] (.4,-.5) node {$P'$};

\draw[fill=black] (NULL) circle (0.05) node[below left] {$0$};

\draw[fill=black] (U3) circle (0.05) node[below left] {$u_3$};
\draw[fill=black] (U5) circle (0.05) node[left] {$u_5$};
\draw[fill=black] (U6) circle (0.05) node[below right] {$u_6$};

\draw[blue, fill=blue] (V1) circle (0.05) node[below] {$v_1$};
\draw[blue, fill=blue] (V2) circle (0.05) node[left] {$v_2$};
\draw[blue, fill=blue] (V3) circle (0.05) node[below left] {$v_3$};
\draw[blue, fill=blue] (V5) circle (0.05) node[right] {$v_5\!$};

\draw[fill=black] (UU3) circle (0.05) node[below] {$\bar u_3$};
\draw[fill=black] (UU5) circle (0.05) node[right] {$\bar u_5$};
\draw[fill=black] (UU6) circle (0.05) node[below right] {$\!\bar u_6$};
\end{tikzpicture}
\end{center}
\caption{$P= \conv\{u_1,u_3,u_4,u_2\}$, $(1+\eps)P= \conv\{\bar u_1,\bar u_3,\bar u_4,\bar u_2\}$ and a ``cut'' $H_0$ are shown on the left. $\V=\{v_1,v_3,v_4,v_2\}$ is a strong $\eps$-approximate V-representation of $P$. Algorithm~\ref{alg_basic_cut} computes $\V'=\{v_1,v_3,v_5,v_2\}$, displayed on the right, which is an $\eps$-approximate V-representation of $P$, but not a strong one. Although $P'$ coincides with $P$ in Figure~\ref{fig_fail}, the algorithm does not fail in the next iteration, see Figure 3.}\label{fig1a}
\end{figure}

\begin{figure}[hpt]
	\def\scalefactor{.75}	
	\def\spacebetweenpictures{1cm}
\begin{center}
\begin{tikzpicture}[scale=\scalefactor,every node/.style={scale=\scalefactor}]
	\figurecoordinates

\draw[white,fill=white] (NE) circle (0.01);
\draw[white,fill=white] (SW) circle (0.01);

\draw [opacity=0.1, green!80!white, fill opacity=0.1,fill=green!80!white] (A1) -- (A2) -- (AA2) -- (AA1) -- cycle;
\draw [opacity=0.1, brown!80!white, fill opacity=0.1,fill=brown!80!white] (C1) -- (C2) -- (CC2) -- (CC1) -- cycle;
\draw [opacity=0.1, yellow!80!white, fill opacity=0.1,fill=yellow!80!white] (D1) -- (D2) -- (DD2) -- (DD1) -- cycle;
\draw [opacity=0.1, red!80!white, fill opacity=0.1,fill=red!80!white] (E1) -- (E2) -- (EE2) -- (EE1) -- cycle;

\draw[green!50!black] (A1) -- (A2);
\draw[green!50!black,dashed] (AA1) -- (AA2);
\draw[brown] (C1) -- (C2);
\draw[brown,dashed] (CC1) -- (CC2);
\draw[yellow!50!black] (D1) -- (D2);
\draw[yellow!50!black,dashed] (DD1) -- (DD2);
\draw[red] (E1) -- (E2);
\draw[red,dashed] (EE1) -- (EE2);

\draw[fill=gray!40!white,very thick] (U5) -- (U6) -- (U3) --  cycle;
\draw[blue,very thick] (V1) -- (V3) -- (V5) -- (V2) --  cycle;

\draw [opacity=0.1, red!50!yellow, fill opacity=0.1,fill=red!50!yellow] (F1) -- (F2) -- (FF2) -- (FF1) -- cycle;
\draw[red!50!yellow] (F1) -- (F2);
\draw[red!50!yellow,dashed] (FF1) -- (FF2);

\draw[black!80!white] (.4,-.5) node {$P$};
\draw[red!50!yellow] (2,7/2) node {$H_0$};

\draw[fill=black] (NULL) circle (0.05) node[below left] {$0$};

\draw[fill=black] (U3) circle (0.05) node[below left] {$u_3$};
\draw[fill=black] (U5) circle (0.05) node[left] {$u_5$};
\draw[fill=black] (U6) circle (0.05) node[below right] {$u_6$};
\draw[red!50!yellow,fill=red!50!yellow] (U7) circle (0.05) node[below] {$u_7$};
\draw[red!50!yellow,fill=red!50!yellow] (U8) circle (0.05) node[right] {$u_8\!$};

\draw[blue, fill=blue] (V1) circle (0.05) node[below] {$v_1$};
\draw[blue, fill=blue] (V2) circle (0.05) node[left] {$v_2$};
\draw[blue, fill=blue] (V3) circle (0.05) node[below left] {$v_3$};
\draw[blue, fill=blue] (V5) circle (0.05) node[left] {$v_5\!$};
\draw[blue,fill=red!50!yellow] (V7) circle (0.05) node[right] {$\!v_7$};
\draw[blue,fill=red!50!yellow] (V8) circle (0.05) node[left] {$v_8\!$};

\draw[fill=black] (UU3) circle (0.05) node[below] {$\bar u_3$};
\draw[fill=black] (UU5) circle (0.05) node[right] {$\bar u_5$};
\draw[fill=black] (UU6) circle (0.05) node[below right] {$\!\bar u_6$};
\draw[red!50!yellow,fill=red!50!yellow] (UU7) circle (0.05) node[above] {$\bar u_7$};
\draw[red!50!yellow,fill=red!50!yellow] (UU8) circle (0.05) node[above left] {$\bar u_8$};
\end{tikzpicture}
\begin{tikzpicture}[scale=\scalefactor,every node/.style={scale=\scalefactor}]
\figurecoordinates

\draw[white,fill=white] (NE) circle (0.01);
\draw[white,fill=white] (SW) circle (0.01);

\draw [opacity=0, fill opacity=0.1,fill=green!80!white] (A1) -- (A2) -- (AA2) -- (AA1) -- cycle;
\draw [opacity=0, fill opacity=0.1,fill=brown!80!white] (C1) -- (C2) -- (CC2) -- (CC1) -- cycle;
\draw [opacity=0, fill opacity=0.1,fill=yellow!80!white] (D1) -- (D2) -- (DD2) -- (DD1) -- cycle;
\draw [opacity=0, fill opacity=0.1,fill=red!80!white] (E1) -- (E2) -- (EE2) -- (EE1) -- cycle;
\draw [opacity=0, fill opacity=0.1,fill=red!50!yellow] (F1) -- (F2) -- (FF2) -- (FF1) -- cycle;

\draw[green!50!black] (A1) -- (A2);
\draw[green!50!black,dashed] (AA1) -- (AA2);
\draw[brown] (C1) -- (C2);
\draw[brown,dashed] (CC1) -- (CC2);
\draw[yellow!50!black] (D1) -- (D2);
\draw[yellow!50!black,dashed] (DD1) -- (DD2);
\draw[red] (E1) -- (E2);
\draw[red,dashed] (EE1) -- (EE2);
\draw[red!50!yellow] (F1) -- (F2);
\draw[red!50!yellow,dashed] (FF1) -- (FF2);

\draw[fill=gray!40!white,very thick] (U6) -- (U3) -- (U7) -- (U8) --  cycle;
\draw[blue, very thick] (V1) -- (V3) -- (V7) -- (V8) --  cycle;

\draw[black!80!white] (.4,-.5) node {$P'$};

\draw[fill=black] (NULL) circle (0.05) node[below left] {$0$};

\draw[fill=black] (U3) circle (0.05) node[below left] {$u_3$};
\draw[fill=black] (U6) circle (0.05) node[below right] {$u_6$};
\draw[fill=black] (U7) circle (0.05) node[below] {$u_7$};
\draw[fill=black] (U8) circle (0.05) node[right] {$u_8\!$};

\draw[blue, fill=blue] (V1) circle (0.05) node[below] {$v_1$};
\draw[blue, fill=blue] (V3) circle (0.05) node[below left] {$v_3$};
\draw[blue, fill=blue] (V7) circle (0.05) node[right] {$v_7\!$};
\draw[blue, fill=blue] (V8) circle (0.05) node[left] {$v_8\!$};

\draw[fill=black] (UU3) circle (0.05) node[below] {$\bar u_3$};
\draw[fill=black] (UU6) circle (0.05) node[below right] {$\!\bar u_6$};
\draw[fill=black] (UU7) circle (0.05) node[above] {$\bar u_7$};
\draw[fill=black] (UU8) circle (0.05) node[above left] {$\bar u_8$};
\end{tikzpicture}
\end{center}
\caption{We start a new iteration with the result of Figure~\ref{fig1a}. $P$ is the same as $P$ in Figure~\ref{fig_fail}. Because of the redundant inequality (brown), which is missing in Figure~\ref{fig_fail}, Algorithm~\ref{alg_basic_cut} does not fail here. The redundant inequality makes the line segment between $v_1$ and $v_2$ ``observable'' for the algorithm. Thus, in contrast to Figure~\ref{fig_fail}, the point $v_8$ is added to $\V'$.}\label{fig1b}
\end{figure}
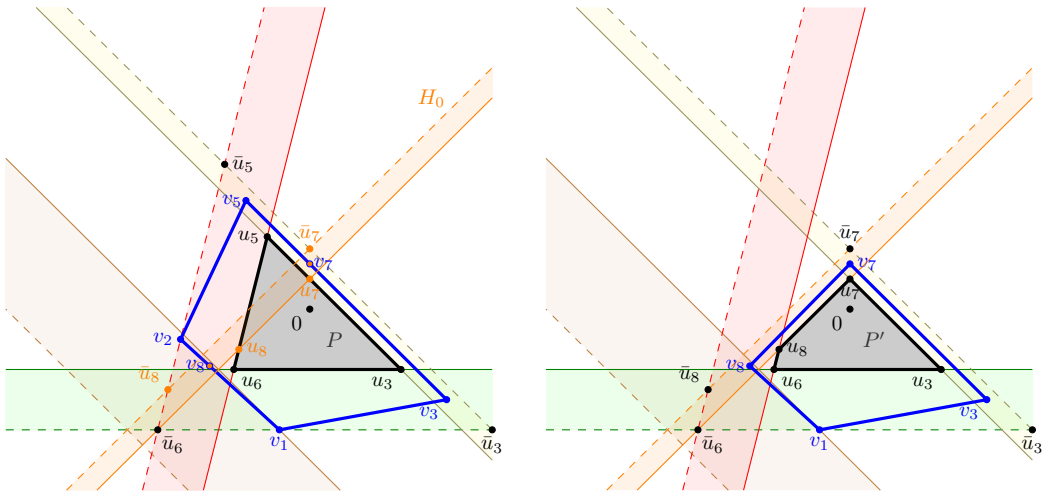

\begin{remark}\label{rem_trivial}
	Note that trivial solution methods exist for the approximate vertex enumeration problem. If the condition in line \ref{line_cond_edge} of Algorithm \ref{alg_basic_cut} is omitted (always true), we obtain a correct algorithm. Such a variant, however, is not ``practically relevant'' as it computes way too many vertices.  
\end{remark}

The following two conditions can be used to prove a valid variant of Algorithm~\ref{alg_basic_cut}. As they involve the vertices of $P$, the computation of which requires to solve the exact vertex enumeration problem, they are of theoretical interest only. Under these conditions Algorithm~\ref{alg_basic_cut} maintains the property of $\V$ being a strong $\eps$-approximate V-representation of $P$. The aim of these investigations is to get an idea under which conditions this invariance property can fail.

Given a strong $\eps$-approximate V-representation $\V$ of $P$ and a half-space $H_\leq \colonequals\{x \mid h^T x \leq 1\}$, $\V$ is said to be {\em $h$-correct} if for every vertex $u \in P$ there exists $v \in \V$ such that $u$ is covered by $v$ and
\begin{enumerate}[({A}1)]
	\item\label{a1} $u \in H_+ \cup H_0 \implies v \in H_+ \cup H_0$,
	\item\label{a2} $u \in H_- \implies v \in H_- \cup H_0$. 
\end{enumerate}

For $\prec \in \{>, \geq, =,\neq, <, \leq\}$ and $u \in \R^d$, we define $J'_\prec(u) \colonequals J_\prec(u) \cup \{m+1\}$ if $h^T u \prec 1$ and $J'_\prec(u) \colonequals J_\prec(u)$ otherwise. This corresponds to adding $h^T$ as the $(m+1)$-th row to the matrix $A \in \R^{m \times d}$ in the definition of $P$. Thus Proposition~\ref{prop_cover2} states that a vertex $u$ of $P'$ is covered by $v \in \R^d$ if and only of $J'_=(u) \subseteq J'_\geq(v)$. 

\begin{theorem} \label{th_basic_cut} If the input $\V$ is an $h$-correct strong $\eps$-approximate V-representation of $P$, then Algorithm~\ref{alg_basic_cut} computes a strong $\eps$-approximate V-representation $\V'$ of $P'$. 
\end{theorem}
\begin{proof}
Let $u$ be a vertex of $P'$. We need to show that the output $\V'$ contains some $v \in (1+\eps)P'$ such that $J'_=(u) \subseteq J'_\geq(v)$.

If $u \in H_-$, $u$ is also a vertex of $P$. There is $v \in \V$ such that $v$ covers $u$, that is, $J_=(u) \subseteq J_\geq(v)$. Since $\V$ is assumed to be $h$-correct, we have $v \in H_- \cup H_0$. Hence $v$ is not cut off in line \ref{cutoff} of the algorithm. We have $v \in (1+\eps)P'$ and $J'_=(u) = J_=(u) \subseteq J_\geq(v) \subseteq J'_\geq(v)$.

Let $u \in H_0$. Then $u$ belongs to an edge of $P$ with endpoints $u_1 \in H_-$, $u_2 \in H_0 \cup H_+$, i.e.\ $u_1$ and $u_2$ are vertices of $P$ and there is $\lambda \in [0,1)$ such that $u=\lambda u_1 + (1-\lambda) u_2$.
The assumption of $\V$ being $h$-correct implies that there are cover points $v_1 \in \V$ of $u_1$ with $v_1 \in H_- \cup H_0$ and $v_2 \in \V$ of $u_2$ with $v_2 \in H_0 \cup H_+$. If $v_1 \in H_0$ or $v_2 \in H_0$, $u$ is covered by $v_1$ or $v_2$ since $J_=(u) \subseteq J_=(u_i)$, $(i=1,2)$. In the remaining case we have $v_1 \in H_-$ and $v_2 \in H_+$. As shown above, the condition in line \ref{line_cond_edge} is satisfied, which implies that the algorithm adds a point $v \in \conv\{v_1,v_2\} \cap H_0$ to $\V$. We have
\begin{equation}\label{eq_ha}
	J_=(u) \subseteq J_=(u_1) \cap J_=(u_2) \subseteq J_\geq(v_1) \cap J_\geq(v_2) \subseteq J_\geq(v),
\end{equation}
where the first inclusion holds as $u$ is on an edge with endpoints $u_1,u_2$, the second inclusion holds as $v_1$ covers $u_1$ and $v_2$ covers $u_2$, and the third inclusion holds as $v \in \conv\{v_1,v_2\}$. Moreover, $v \in H_0 \cap (1+\eps)P$ implies $v \in (1+\eps) P'$. From $v \in H_0$ we deduce $m+1 \in J'_\geq(v)$. Since $u \in H_\leq \cap H_0$ we get $m+1 \in J'_=(u)$. 
By \eqref{eq_ha}, we have $J'_=(u) = J_=(u) \cup \{m+1\} \subseteq J_\geq(v) \cup \{m+1\} = J'_\geq(v)$. 
\end{proof}

In Figure~\ref{fig1a} we see that condition (A1) cannot be omitted in Theorem~\ref{th_basic_cut}. Condition (A1) is violated since $u_1$ is covered by $v_1$ only, but $u_1 \in H_0$ and $v_1\in H_-$. We close this section with an example showing that condition (A2) cannot be omitted in Theorem~\ref{th_basic_cut}.

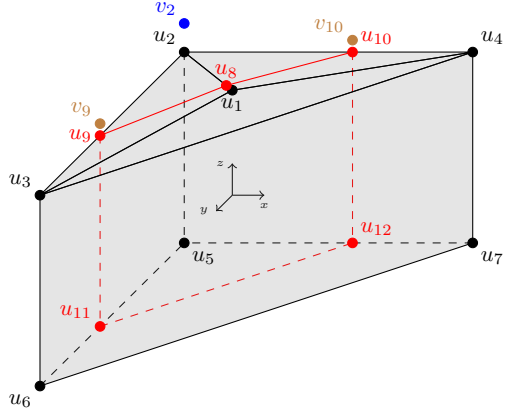
\begin{figure}[ht]
	\begin{center}
\begin{tikzpicture}[line join=bevel, x=40,z=-20,y=40, scale=0.9, every node/.style={scale=0.85}] 
\coordinate (NULL) at (0,0,0);
\coordinate (X) at (1/3,0,0);
\coordinate (Y) at (0,0,1/3);
\coordinate (Z) at (0,1/3,0);
\coordinate (U1) at (0,1.1,0);
\coordinate (U2) at (-1,1,-1);
\coordinate (U3) at (-1,1,2);
\coordinate (U4) at (2,1,-1);
\coordinate (U5) at (-1,-1,-1);
\coordinate (U6) at (-1,-1,2);
\coordinate (U7) at (2,-1,-1);
\coordinate (V2) at (-1,1.3,-1);
\coordinate (V9) at (-1,1+5/40,3/4);
\coordinate (V10) at (3/4,1+5/40,-1);

\coordinate (U8) at (-1/8,1+7/80,-1/8);
\coordinate (U9) at (-1,1,3/4);
\coordinate (U10) at (3/4, 1,-1);
\coordinate (U11) at (-1,-1,3/4);
\coordinate (U12) at (3/4,-1,-1);

\draw[->] (NULL) -- (X);
\draw[->] (NULL) -- (Y);
\draw[->] (NULL) -- (Z);
\draw[fill=black] (X) node[below] {{\tiny$x$}};
\draw[fill=black] (Y) node[left] {{\tiny$y$}};
\draw[fill=black] (Z) node[left] {{\tiny$z$}};
\draw[dashed] (U7) -- (U5);
\draw[dashed] (U2) -- (U5);
\draw[dashed] (U6) -- (U5);
\draw[red,dashed] (U9) -- (U11);
\draw[red,dashed] (U11) -- (U12);
\draw[red,dashed] (U10) -- (U12);
\draw [fill opacity=0.2,fill=black!50!white] (U1) -- (U2) -- (U3) -- cycle;
\draw [fill opacity=0.2,fill=black!50!white](U1) -- (U3) -- (U4) -- cycle;
\draw [fill opacity=0.2,fill=black!50!white](U1) -- (U2) -- (U4) -- cycle;
\draw [fill opacity=0.2,fill=black!50!white](U3) -- (U4) -- (U7) -- (U6) -- cycle;

\draw[red] (U10) -- (U8);
\draw[red] (U8) -- (U9);
\draw[fill=black] (U1) circle (0.05) node[below] {$u_1$};
\draw[fill=black] (U2) circle (0.05) node[above left] {$u_2$};
\draw[fill=black] (U3) circle (0.05) node[above left] {$u_3$};
\draw[fill=black] (U4) circle (0.05) node[above right] {$u_4$};
\draw[fill=black] (U5) circle (0.05) node[below right] {$u_5$};
\draw[fill=black] (U6) circle (0.05) node[below left] {$u_6$};
\draw[fill=black] (U7) circle (0.05) node[below right] {$u_7$};
\draw[blue, fill=blue] (V2) circle (0.05) node[above left] {$v_2$};
\draw[brown, fill=brown] (V9) circle (0.05) node[above left] {$v_9$};
\draw[brown, fill=brown] (V10) circle (0.05) node[above left] {$v_{10}$};

\draw[red, fill=red] (U8) circle (0.05) node[above] {$u_8$};
\draw[red, fill=red] (U9) circle (0.05) node[left] {$u_9$};
\draw[red, fill=red] (U10) circle (0.05) node[above right] {$u_{10}$};
\draw[red, fill=red] (U11) circle (0.05) node[above left] {$u_{11}$};
\draw[red, fill=red] (U12) circle (0.05) node[above right] {$u_{12}$};
\end{tikzpicture}
\end{center}
\caption{The polytope $P$ from Example~\ref{ex_A2} for $\delta = \frac{1}{10}$. For $\eps=3\delta$, the set $\V=\{v_2,u_3,\dots,u_7\}$ provides a strong $\eps$-approximate V-representation of $P$, where $v_2$ covers both $u_1$ and $u_2$. The cut with the half-space $H_\leq$ and the resulting new vertices are shown in red. The set $\V'=\{u_3,u_4,u_6,u_7,v_9,v_{10},u_{11},u_{12}\}$ computed by Algorithm~\ref{alg_basic_cut} is not a strong $\eps$-approximate V-representation of $P'=P\cap H$ since the vertices $u_1$ and $u_8$ of $P'$ are not covered.}\label{3dex}
\end{figure}
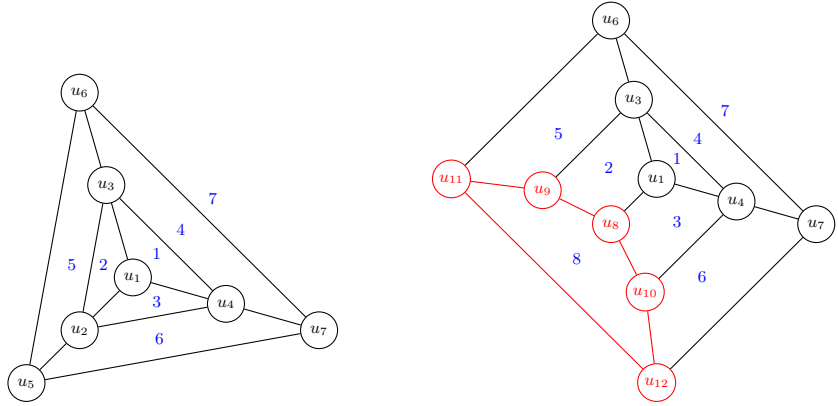
\begin{figure}[ht]
	\begin{center}
\begin{tikzpicture}[scale=.35,every node/.style={scale=0.7}]
	\coordinate (U1) at (0,0);
	\coordinate (U2) at (-2,-2);
	\coordinate (U3) at (-1,3.5);
	\coordinate (U4) at (3.5,-1);
	\coordinate (U5) at (-4,-4);
	\coordinate (U6) at (-2,7);
	\coordinate (U7) at (7,-2);
	\coordinate (H1) at (.9,.9);
	\coordinate (H2) at (-1.1,.5);
	\coordinate (H3) at (.9,-.9);
	\coordinate (H4) at (1.8,1.8);
	\coordinate (H5) at (-2.3,.5);
	\coordinate (H6) at (1,-2.3);
	\coordinate (H7) at (3,3);
	
	\node[shape=circle,draw=black] (N1) at (U1) {\small $u_1$};
	\node[shape=circle,draw=black] (N2) at (U2) {\small $u_2$};
	\node[shape=circle,draw=black] (N3) at (U3) {\small $u_3$};
	\node[shape=circle,draw=black] (N4) at (U4) {\small $u_4$};
	\node[shape=circle,draw=black] (N5) at (U5) {\small $u_5$};
	\node[shape=circle,draw=black] (N6) at (U6) {\small $u_6$};
	\node[shape=circle,draw=black] (N7) at (U7) {\small $u_7$};

	\draw
	(N1) edge[-] (N2)
	(N1) edge[-] (N3)
	(N1) edge[-] (N4)
	(N2) edge[-] (N5)
	(N3) edge[-] (N6)
	(N4) edge[-] (N7)
	(N2) edge[-] (N3)
	(N3) edge[-] (N4)
	(N4) edge[-] (N2)
	(N5) edge[-] (N6)
	(N6) edge[-] (N7)
	(N7) edge[-] (N5);
	\draw[blue] (H1) node {{\small$1$}};
	\draw[blue] (H2) node {{\small$2$}};
	\draw[blue] (H3) node {{\small$3$}};
	\draw[blue] (H4) node {{\small$4$}};
	\draw[blue] (H5) node {{\small$5$}};
	\draw[blue] (H6) node {{\small$6$}};
	\draw[blue] (H7) node {{\small$7$}};
\end{tikzpicture}
\hspace{1cm}
\begin{tikzpicture}[scale=.3, every node/.style={scale=0.7}]
	\coordinate (U1) at (0,0);
	\coordinate (U3) at (-1,3.5);
	\coordinate (U4) at (3.5,-1);
	\coordinate (U6) at (-2,7);
	\coordinate (U7) at (7,-2);
	\coordinate (U8) at (-2,-2);
	\coordinate (U9) at (-5,-.5);
	\coordinate (U10) at (-.5,-5);
	\coordinate (U11) at (-9,0);
	\coordinate (U12) at (0,-9);

	\coordinate (H1) at (.9,.9);
	\coordinate (H2) at (-2.1,.5);
	\coordinate (H3) at (.9,-1.9);
	\coordinate (H4) at (1.8,1.8);
	\coordinate (H5) at (-4.3,2);
	\coordinate (H6) at (2,-4.3);
	\coordinate (H7) at (3,3);
	\coordinate (H8) at (-3.5,-3.5);
	
	\node[shape=circle,draw=black] (N1) at (U1) {\small $u_1$};
	\node[shape=circle,draw=black] (N3) at (U3) {\small $u_3$};
	\node[shape=circle,draw=black] (N4) at (U4) {\small $u_4$};
	\node[shape=circle,draw=black] (N6) at (U6) {\small $u_6$};
	\node[shape=circle,draw=black] (N7) at (U7) {\small $u_7$};
	\node[red,shape=circle,draw=red] (N8) at (U8) {\small $u_8$};
	\node[red,shape=circle,draw=red] (N9) at (U9) {\small $u_9$};
	\node[red,shape=circle,draw=red] (N10) at (U10) {\small $\!u_{10}\!$};
	\node[red,shape=circle,draw=red] (N11) at (U11) {\small $\!u_{11}\!$};
	\node[red,shape=circle,draw=red] (N12) at (U12) {\small $\!u_{12}\!$};

	\draw
	(N1) edge[-] (N3)
	(N1) edge[-] (N4)
	(N1) edge[-] (N8)
	(N3) edge[-] (N6)
	(N4) edge[-] (N7)
	(N4) edge[-] (N10)
	(N3) edge[-] (N9)
	(N3) edge[-] (N4)
	(N8) edge[red,-] (N9)
	(N8) edge[red,-] (N10)
	(N12) edge[red,-] (N10)
	(N9) edge[red,-] (N11)
	(N6) edge[-] (N7)
	(N6) edge[-] (N11)
	(N7) edge[-] (N12)
	(N11) edge[red,-] (N12);
	\draw[blue] (H1) node {{\small$1$}};
	\draw[blue] (H2) node {{\small$2$}};
	\draw[blue] (H3) node {{\small$3$}};
	\draw[blue] (H4) node {{\small$4$}};
	\draw[blue] (H5) node {{\small$5$}};
	\draw[blue] (H6) node {{\small$6$}};
	\draw[blue] (H7) node {{\small$7$}};
	\draw[blue] (H8) node {{\small$8$}};
\end{tikzpicture}
\end{center}
\caption{Schlegel diagrams of $P$ (left) and $P'$ (right) from Example~\ref{ex_A2} showing the inequality-vertex incidence information.}\label{3dexschlegel}
\end{figure}

\begin{example} \label{ex_A2}
For some $\delta > 0$ we consider the polytope $P = \{x \in \R^3 \mid A x \leq \one\}$ with data
\begin{equation*}
	\begingroup
	\renewcommand*{\arraystretch}{1.2}
	A=\begin{pmatrix*}[r]
		\frac{\delta}{1+\delta} &\frac{\delta}{1+\delta} & \frac{1}{1+\delta}\\
	 -\frac{\delta}{1+\delta} & 0 & \frac{1}{1+\delta} \\
	 0 & -\frac{\delta}{1+\delta} & \frac{1}{1+\delta} \\
	 1 & 1 & 0 \\
	 -1 & 0 & 0 \\
	 0 & -1 & 0 \\
	 0 & 0 & -1
	\end{pmatrix*},
	\endgroup
\end{equation*}
which is illustrated in Figure~\ref{3dex}. The vertices of $P$ are
\begin{equation*}
	u_1 \!=\!\! \begin{pmatrix*}[c]
		0 \\ 0 \\ 1 + \delta
	\end{pmatrix*} \!\!,\;
	u_2 \!=\!\!\begin{pmatrix*}[r]
		-1 \\ -1 \\ 1
	\end{pmatrix*} \!\!,\;
	u_3 \!=\!\! \begin{pmatrix*}[r]
		-1  \\ 2 \\ 1
	\end{pmatrix*} \!\!,\;
	u_4 \!=\!\! \begin{pmatrix*}[r]
		2  \\ -1 \\ 1
	\end{pmatrix*} \!\!,\;
	u_5 \!=\!\! \begin{pmatrix*}[r]
		-1  \\ -1 \\ -1
	\end{pmatrix*} \!\!,\;
	u_6 \!=\!\! \begin{pmatrix*}[r]
		-1  \\ 2 \\ -1
	\end{pmatrix*} \!\!,\;
	u_7 \!=\!\! \begin{pmatrix*}[r]
		2  \\ -1 \\ -1
	\end{pmatrix*}\!\!.
\end{equation*}
We have
\begin{equation*}
	\begin{array}{llll}
	J_=(u_1) = \{1,2,3\}, &
	J_=(u_2) = \{2,3,5,6\}, &
	J_=(u_3) = \{1,2,4,5\}, &
	J_=(u_4) = \{1,3,4,6\}, \\
	J_=(u_5) = \{5,6,7\}, &
	J_=(u_6) = \{4,5,7\}, &
	J_=(u_7) = \{4,6,7\}, &
	\end{array}
\end{equation*}
compare the first Schlegel diagram in Figure~\ref{3dexschlegel}. Let $\V=\{v_2,\dots,v_7\}$ where we set 
\begin{equation*}
	v_2 = \begin{pmatrix*}[c]
		-1 \\ -1 \\ 1+3 \delta
	\end{pmatrix*}, \quad v_3 = u_3,\quad \dots \quad,  v_7 = u_7.
\end{equation*}
We have
\begin{equation*}
	J_=(v_2) = \{1, 5, 6\}, \qquad J_\geq(v_2) = \{ 1,2,3,5,6\}.
\end{equation*} 
One can easily verify that
\begin{equation*}
	v_2 \in (1+3\delta) P.
\end{equation*}
Thus, for $\eps = 3 \delta$, $\V$ is a strong $\eps$-approximate V-representation of $P$. Note that $v_2$ covers both $u_1$ and $u_2$.
Now add to $A$ the row 
\begin{equation*}
	A_8 = \begin{pmatrix*}[r]
		-4&-4&0
	\end{pmatrix*}
\end{equation*} 
and let $P'$ be the corresponding polyhedron, i.e.\ the intersection of $P$ and the half-space $H_\leq = \{x \in \R^3 \mid A_8 x \leq 1\}$. The vertices of $P'$ are $u_1, u_3, u_4, u_6, u_7$ as defined above as well as the new vertices 
\begin{equation*}
	\begingroup
	\renewcommand*{\arraystretch}{1.2}
	u_8\!=\!\! \begin{pmatrix*}[c]
		- \frac{1}{8} \\ -\frac{1}{8} \\ 1 + \frac{7}{8}\delta
	\end{pmatrix*} \!\!,\;
	u_9 \!=\!\!\begin{pmatrix*}[r]
		-1 \\ \frac{3}{4} \\ 1
	\end{pmatrix*} \!\!,\;
	u_{10} \!=\!\! \begin{pmatrix*}[r]
		\frac{3}{4} \\ -1 \\ 1
	\end{pmatrix*} \!\!,\;
	u_{11} \!=\!\! \begin{pmatrix*}[r]
		-1 \\ \frac{3}{4} \\ -1
	\end{pmatrix*} \!\!,\;
	u_{12} \!=\!\! \begin{pmatrix*}[r]
		\frac{3}{4} \\ -1 \\ -1
	\end{pmatrix*}
	\endgroup
\end{equation*}
with indices
\begin{equation*}
	\begin{array}{lll}
	J_=(u_8) = \{2,3,8\}, &
	J_=(u_9) = \{2,5,8\}, &
	J_=(u_{10}) = \{3,6,8\}, \\
	J_=(u_{11}) = \{5,7,8\}, &
	J_=(u_{12}) = \{6,7,8\},
	\end{array}
\end{equation*}
compare the second Schlegel diagram in Figure~\ref{3dexschlegel}.
An update $\V'$ of the set $\V$ by Algorithm~\ref{alg_basic_cut} contains the points $v_3, v_4, v_6, v_7$. The points $v_2$ and $v_5$ are cut off by the algorithm if and only if $\eps < 7$, thus let us assume $3 \delta = \eps < 7$. A valid choice of  new points added to $\V$ is $v_{11} = u_{11}$, $v_{12} = u_{12}$,
\begin{equation*}
	\begingroup
	\renewcommand*{\arraystretch}{1.2}
	v_9=\begin{pmatrix*}[c]
		-1 \\ \frac{3}{4} \\ 1 + \frac{5}{4} \delta
	\end{pmatrix*} \in \conv\{v_2,v_3\} \cap H_0\quad \text{and}\quad
	v_{10}= \begin{pmatrix*}[c]
		\frac{3}{4} \\ -1 \\ 1 + \frac{5}{4} \delta
	\end{pmatrix*}\in \conv\{v_2,v_4\} \cap H_0. 
	\endgroup
\end{equation*}
For $i \in \{3,4,6,7,11,12\}$ we have $J_\geq(v_i) = J_=(u_i)$ as those points $v_i=u_i$ are vertices of $P'$. Moreover we have
\begin{equation*}
	J_\geq(v_9) = \{1,2,5,8\}, \qquad J_\geq(v_{10})=\{1,3,6,8\}.
\end{equation*}
We see that there is no $v \in \V'$ with $J_=(u_1) = \{1,2,3\} \subseteq J_\geq (v)$, i.e.\ the vertex $u_1$ (and likewise the vertex $u_8$) of $P'$ is not covered by some point in $\V'$. Thus $\V'$ is not a strong $\eps$-approximate V-representation of $P'$.
\end{example}

\section{The approximate double description method}\label{sec_addm}

In this section we define an algorithm, called {\em approximate double description method}, which is based on repeated application of the basic cutting scheme of Algorithm \ref{alg_basic_cut}. We make the following modifications:
\begin{enumerate}[(i)]
	\item We formulate the algorithm using a graph $G=(V,E)$ the reason of which is to have a formulation which is convenient to be compared with another algorithm in the next section. The points in $\R^d$ are expressed by a coordinate function $c:V\to \R^d$.
	\item Instead of $J_\geq(v)$ we use an index set $I(v)$ which is recursively defined.
	\item We allow a more general partitioning of the space $\R^d$, because this will be useful in Section \ref{sec_imprec} where imprecise arithmetic is discussed.
\end{enumerate}

\begin{algorithm}[ht]
\DontPrintSemicolon
\SetKwInOut{Input}{input}\SetKwInOut{Output}{output}
\Input{$A \in \R^{m\times d}$ such that $P \colonequals  \{x \mid A x \leq \one\}$ is a polytope and $S \colonequals  \{x \mid A_{[d+1]} x \leq \one\}$ is a simplex; tolerance $\varepsilon \geq 0$}
\Output{some graph $G=(V,E)$ and a coordinate function $c: V \to \R^d$ such that $P \subseteq \conv c(V) \subseteq (1+\varepsilon) P$}
\Begin{
	initialize the graph $G = G(V,E)$ by $K_{d+1}$ with nodes $V=\{v_1,\dots,v_{d+1}\}$\;
	compute vertices $\{u_1,\dots,u_{d+1}\}$ of $(1+\frac{\varepsilon}{2})S$ and set $c(v_i) \leftarrow u_i$ for $i \in \{1,\dots,d+1\}$\;	
    \For{$i\leftarrow d+2$ {\textbf{to}} $m$}{
		partition $V$ into disjoint sets $V_- \neq \emptyset$, $V_0$, $V_+$ such that
		$V_- \subseteq \{v \in V \mid A_i c(v) < 1 + \frac{\varepsilon}{2} \}$,
		$V_+ \subseteq \{v \in V \mid A_i c(v) > 1 + \frac{\varepsilon}{2} \}$,
		$V_0 \subseteq \{v \in V \mid 1 \leq A_i c(v) \leq 1 + \varepsilon \}$\;	
		\For{$e = uw \in E$ satisfying $u \in V_-$ and $w \in V_+$}{
			add a new node $v=v(u,w)$ to $V_0$ and $V$\;	
			$I(v) \leftarrow I(u)\cap I(w)$\; 
			define $c(v)$ by choosing a point in the line segment between $c(u)$ and $c(w)$ such that $1 \leq A_i c(v) \leq 1 + \varepsilon$\;	
		}
		\For{$u \in V_0$}{
			$I(u) \leftarrow I(u) \cup \{i\}$
		}
		remove $V_+$ from $V$\;
		$E \leftarrow \{ uw \mid\; u,v \in V,\; |I(u) \cap I(w)| \geq d-1 \}$\;			  
    }
}
\caption{Approximate double description method}
\label{addm}
\end{algorithm}

The approximate double description method is defined in Algorithm~\ref{addm}. Since the correctness result for the basic cutting scheme formulated in Theorem \ref{th_basic_cut} of the previous section is of only theoretical nature, we refrain from extending these results to the more general setting. Instead we focus in the next sections on the $2$- and $3$-dimensional case, where correctness can be shown by graph theoretical methods without any impracticable conditions like {\em h-correctness}.

\section{Algorithms for 2- and 3-polytopes} \label{sec_23}

In this section we prove correctness of the approximate double description method  for polytopes in dimensions $2$ and $3$. To this end we introduce another algorithm, called the {\em shortcut algorithm}, which is shown to be related to the approximate double description method in the following sense: If the shortcut algorithm solves the approximate vertex enumeration problem correctly then so does the approximate double description method. So we will prove correctness of the shortcut algorithm and obtain correctness of the approximate double description problem as a corollary. The numerical results in the next section show that the shortcut algorithm is not only of theoretical but also of practical importance. 

The shortcut algorithm is based on a construction of a planar graph. This construction starts with the complete graph $K_{d+1}$ and maintains planarity. Since $K_{d+1}$ is not planar for $d \geq 4$, this method is restricted to dimension $d \leq 3$.

Let us recall some basic concepts from graph theory, where we follow the book by Diestel \cite{Diestel05}. A {\em graph} is a pair $G=(V,E)$, where $E$ (the set of edges) is a subset of $2$-element subsets of $V$ (the set of vertices). For $e = \{u,v\} \in E$ we write $e=uv$.
The graph $G'=(V',E')$ is a {\em subgraph} of $G=(V,E)$ if $V' \subseteq V$ and $E' \subseteq E$. A {\em walk} in a graph $G$ is a non-empty alternating sequence $v_0,e_0,v_1,e_1,\dots,e_{k-1},v_k$ of vertices and edges in $G$ such that $e_i = v_iv_{i+1}$ for all $i<k$. If $v_0 = v_k$, the walk is called {\em closed}. 
A walk with all vertices being distinct is called a {\em path}.
A maximal connected subgraph of $G$ is called a {\em component}.

A {\em plane graph} is a graph $G=(V,E)$ together with an embedding of $G$ into the sphere $S^2$. This means that (i) the vertices are points in $S^2$, (ii) every edge is an arc between two vertices, (iii) different edges have different sets of endpoints, (iv) the interior of an edge contains no vertex and no point of another edge. A plane graph can be seen as a drawing of an (abstract) graph $G$ on the sphere. A graph $G=(V,E)$ is called {\em planar} if such an embedding exists. A plane graph partitions the sphere into regions, called {\em faces} which are bounded by arcs. According to \cite[Lemma 4.2.2]{Diestel05} an arc $uw \in E$ lies on the frontier of either one or two faces. These faces are called {\em incident} with $uw$. A {\em bounding walk} of a face $f$ is a closed walk $W$ in $G$ such that
\begin{enumerate}[(i)]
	\item an edge $e$ is at most twice contained in $W$,
	\item an edge $e$ contained exactly once in $W$ is incident with $f$ and with another face $f'$,
	\item an edge $e$ contained exactly twice in $W$ is incident with $f$ only,
	\item $W$ is maximal with these properties.
\end{enumerate} 
The concept of bounding walk is illustrated in Figure \ref{fig_bw}.

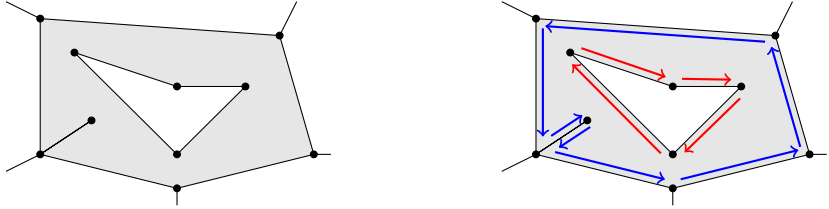
\begin{figure}[ht]
	\begin{center}
\begin{tikzpicture}[scale=.9]
	\coordinate (V1)  at (0,0);
	\coordinate (V01) at (-1/2,-1/4);
	\coordinate (V2)  at (0,2);
	\coordinate (V02) at (-1/2,9/4);
	\coordinate (V3)  at (7/2,7/4);
	\coordinate (V03) at (15/4,9/4);
	\coordinate (V4)  at (4,0);
	\coordinate (V04) at (17/4,0);
	\coordinate (V5)  at (2,-1/2);
	\coordinate (V05) at (2,-3/4);
	\coordinate (V6)  at (3/4,1/2);	
	\coordinate (VA)  at (3,1);
	\coordinate (VB)  at (2,0);
	\coordinate (VC)  at (1/2,3/2);
	\coordinate (VD)  at (2,1);	
	\draw [fill=gray!20!white] (V1) -- (V6) -- (V1) -- (V2) -- (V3) -- (V4) -- (V5) -- (V1);	
	\draw [fill=white] (VA) -- (VB) -- (VC) -- (VD) -- (VA);	
	\draw (V1) -- (V01);
	\draw (V2) -- (V02);
	\draw (V3) -- (V03);
	\draw (V4) -- (V04);
	\draw (V5) -- (V05);
	\draw[fill=black] (V1) circle (0.05);
	\draw[fill=black] (V2) circle (0.05);
	\draw[fill=black] (V3) circle (0.05);	
	\draw[fill=black] (V4) circle (0.05);
	\draw[fill=black] (V5) circle (0.05);
	\draw[fill=black] (V6) circle (0.05);	
	\draw[fill=black] (VA) circle (0.05);
	\draw[fill=black] (VB) circle (0.05);
	\draw[fill=black] (VC) circle (0.05);	
	\draw[fill=black] (VD) circle (0.05);
\end{tikzpicture}
\hspace{2cm}
\begin{tikzpicture}[scale=.9]
	\coordinate (V1)  at (0,0);
	\coordinate (V01) at (-1/2,-1/4);
	\coordinate (V2)  at (0,2);
	\coordinate (V02) at (-1/2,9/4);
	\coordinate (V3)  at (7/2,7/4);
	\coordinate (V03) at (15/4,9/4);
	\coordinate (V4)  at (4,0);
	\coordinate (V04) at (17/4,0);
	\coordinate (V5)  at (2,-1/2);
	\coordinate (V05) at (2,-3/4);
	\coordinate (V6)  at (3/4,1/2);	
	\coordinate (VA)  at (3,1);
	\coordinate (VB)  at (2,0);
	\coordinate (VC)  at (1/2,3/2);
	\coordinate (VD)  at (2,1);	
	\draw [fill=gray!20!white] (V1) -- (V6) -- (V1) -- (V2) -- (V3) -- (V4) -- (V5) -- (V1);	
	\draw [fill=white] (VA) -- (VB) -- (VC) -- (VD) -- (VA);	
	\draw (V1) -- (V01);
	\draw (V2) -- (V02);
	\draw (V3) -- (V03);
	\draw (V4) -- (V04);
	\draw (V5) -- (V05);
	\DoubleLine{3pt}{80}{10pt}{17}{V6}{V1}{->,blue,thick}{}{<-,blue,thick}
	\SingleLine{5pt}{35}{8pt}{21}{V2}{V1}{->,blue,thick}{}
	\SingleLine{5pt}{35}{5pt}{35}{V3}{V2}{->,blue,thick}{}
	\SingleLine{5pt}{35}{5pt}{35}{V4}{V3}{->,blue,thick}{}
	\SingleLine{5pt}{35}{5pt}{35}{V5}{V4}{->,blue,thick}{}
	\SingleLine{8pt}{21}{5pt}{35}{V1}{V5}{->,blue,thick}{}	
	\SingleLine{5pt}{40}{5pt}{35}{VD}{VA}{->,red,thick}{}
	\SingleLine{5pt}{40}{5pt}{35}{VC}{VD}{->,red,thick}{}
	\SingleLine{5pt}{40}{5pt}{35}{VB}{VC}{->,red,thick}{}
	\SingleLine{5pt}{40}{5pt}{35}{VA}{VB}{->,red,thick}{}
	\draw[fill=black] (V1) circle (0.05);
	\draw[fill=black] (V2) circle (0.05);
	\draw[fill=black] (V3) circle (0.05);	
	\draw[fill=black] (V4) circle (0.05);
	\draw[fill=black] (V5) circle (0.05);
	\draw[fill=black] (V6) circle (0.05);	
	\draw[fill=black] (VA) circle (0.05);
	\draw[fill=black] (VB) circle (0.05);
	\draw[fill=black] (VC) circle (0.05);	
	\draw[fill=black] (VD) circle (0.05);
\end{tikzpicture}

\end{center}
\caption{Left: A face $f$ of a plane graph. Right: The two bounding walks of the face $f$.}\label{fig_bw}
\end{figure}

The formulation of the {\em shortcut algorithm} introduced in this section is almost identical for dimension $d=2$ and dimension $d=3$. The only exception is that we have to distinguish between two types of faces, called {\em valid} and {\em invalid faces}, in case of $d=2$, while for $d=3$ all faces are valid. In the shortcut algorithm we use tacitly the following initialization and update rules:
\begin{enumerate}[({I}1)]
	\item[(I1)] One of the two faces of $K_3$ is set to `valid', the other one to `invalid'.
	\item[(I2)] All the four faces of $K_4$ are set to `valid'.
	\item[(U1)] Merging two faces $f_1, f_2$ (by deleting edges) results in a valid face if both $f_1$ and $f_2$ are valid. Otherwise it results into an invalid face.
	\item[(U2)] Splitting an (in)valid face (by adding edges) results in two (in)valid faces.
\end{enumerate}
Thus, in the Algorithms \ref{core_ga} and \ref{ga} below, there will be exactly one invalid face for $d=2$ (invalid faces are not splitted) and all faces are valid for $d=3$.

Algorithm \ref{core_ga} is an abstract variant of the shortcut algorithm. It can be seen as the graph theoretical core of the actual shortcut algorithm, which is introduced later in Algorithm \ref{ga} by specifying a rule for partitioning the set of vertices into three disjoint subsets. The term ``shortcut'' refers to line 10 of Algorithm \ref{core_ga}, where ``shortcuts'' along bounding walks are inserted. The algorithm is illustrated in Figures \ref{ex_core_ga_2} and \ref{ex_core_ga_3}. Algorithm \ref{core_addm} can be seen as the graph theoretical core of the approximate double description method. It is explained by an example in Figure \ref{ex_core_addm}. 

We will show in Theorem \ref{th_rel_core} that Algorithm \ref{core_ga} constructs a subgraph of the graph computed by Algorithm \ref{core_addm}. This connection will be used later in Corollary \ref{cor_rel} to derive correctness of the approximate double description method from correctness of the shortcut algorithm. The latter is proven directly, see Theorems \ref{th_ga_2} and \ref{th_ga_3} below.

\begin{algorithm}[hpt]
\DontPrintSemicolon
\SetKwInOut{Input}{input}\SetKwInOut{Output}{output}
\Input{$d \in \{2,3\}$, $m \in \mathbb{N}$}
\Output{some plane graph $\bar G=(\bar V,\bar E)$}
\Begin{
	initialize the (plane) graph $\bar G = \bar G(\bar V,\bar E)$ by $K_{d+1}$ (taking into account (I1) and (I2)) with nodes $\bar V=\{v_1,\dots,v_{d+1}\}$\;
    \For{$i\leftarrow d+2$ {\textbf{to}} $m$}{
		partition $\bar V$ into disjoint sets $\bar V_- \neq \emptyset$, $\bar V_0$, $\bar V_+$\;
		\For{$e = uw \in \bar E$ satisfying $u \in \bar V_-$ and $w \in \bar V_+$}{
			add a new node $v\colonequals v(u,w)$ to $\bar V_0$ and $\bar V$\;			
			in $\bar E$ replace $uw$ by two new edges $uv$ and $vw$\; 			
		}
		\For{$u_0 u_+, w_0w_+  \in E$ with $u_0,w_0 \in \bar V_0$, $u_0 \neq w_0$, $u_0 w_0 \not\in \bar E$, $u_+,w_+ \in \bar V_+$}{
			\If{there is a walk along a bounding walk of some valid face $f$ from $u_0$ to $w_0$ with all intermediate nodes belonging to $\bar V_+$ or all intermediate nodes belonging to $\bar V_-$}
			{
				$\bar E \leftarrow \bar E \cup \{u_0w_0\}$\;
			} 
		}
		remove all edges incident with some $v \in \bar V_+$ from $\bar E$\;
		remove $\bar V_+$ from $\bar V$\;
		remove multiples of edges from $\bar E$		  
    }
}
\caption{Core of the shortcut algorithm for approximate vertex enumeration}
\label{core_ga}
\end{algorithm}

\begin{algorithm}[hpt]
\DontPrintSemicolon
\SetKwInOut{Input}{input}\SetKwInOut{Output}{output}
\Input{$d \in \{2,3\}$, $m \in \mathbb{N}$}
\Output{some graph $\hat G=(\hat V,\hat E)$}
\Begin{
	initialize the graph $\hat G = \hat G(\hat V,\hat E)$ by $K_{d+1}$ with nodes $\hat V=\{v_1,\dots,v_{d+1}\}$\;
	define an incidence function by setting	$I(v_i) \leftarrow \{1,\dots,d+1\} \setminus \{i\}$ for all $i \in \{1,\dots,d+1\}$\;
    \For{$i\leftarrow d+2$ {\textbf{to}} $m$}{
		partition $\hat V$ into disjoint sets $\hat V_- \neq \emptyset$, $\hat V_0$, $\hat V_+$\;		
		\For{$e = uw \in \hat E$ satisfying $u \in \hat V_-$ and $w \in \hat V_+$}{
			add a new node $v=v(u,w)$ to $\hat V_0$ and $\hat V$\;	
			$I(v) \leftarrow I(u)\cap I(w)$
		}
		\For{$u \in \hat V_0$}{
			$I(u) \leftarrow I(u) \cup \{i\}$
		}
		remove $\hat V_+$ from $\hat V$\;
		$\hat E \leftarrow \{ uv \mid\; u,v \in \hat V,\; |I(u) \cap I(v)| \geq d-1 \}$\;			  
    }
}
\caption{Core of the approximate double description method}
\label{core_addm}
\end{algorithm}

In what follows, we run Algorithms \ref{core_ga} and \ref{core_addm} in parallel and compare the graphs $\bar G$ and $\hat G$ after each (outer) iteration. We assume the {\em same input} for both algorithms. Moreover we use the {\em same partitioning rule}: Whenever $\bar V \subseteq \hat V$, the partitioning rule results in
$$ \bar V_- \subseteq \hat V_-,\quad \bar V_0 \subseteq \hat V_0,\quad \bar V_+ \subseteq \hat V_+.$$ 

\begin{theorem}\label{th_rel_core} For the same input and using the same partitioning rule, Algorithm \ref{core_ga} computes a plane subgraph of the graph computed by
Algorithm \ref{core_addm}.	
\end{theorem}

\begin{proof} In this proof we compare sets computed by Algorithms \ref{core_ga} and \ref{core_addm}. All these comparisons are made either directly before the outer for loop (referred to as ``after initialization'') or at the end of the outer loop with respect to the same iteration index $i$ (referred to as ``after iteration $i$'').

	We show by induction that after each outer iteration:
	\begin{enumerate}[(i)]
		\item $\bar V \subseteq \hat V$,
		\item $\bar E \subseteq \hat E$,
		\item for the vertices $V(f) \subseteq \bar V$ incident with a face $f$ of $\bar G$, 
	$ N_f\colonequals \left|\bigcap\{I(v)\mid v \in V(f)\}\right| \geq d-2$.
	\end{enumerate}

These three statements are obviously true after initialization. Assume the statements hold after iteration $i$. Let $f$ be a face of $\bar G$ after iteration $i+1$. If the vertices $V(f)$ of $f$ are vertices of a face $f_0$ of $\bar G$ after iteration $i$, then the claim follows as incidence sets $I(v)$ can only get bigger during iteration $i+1$. This case covers all new faces which were created by adding edges to $\bar G$ in the second inner loop of Algorithm \ref{core_ga}. In the next case we consider a face $f$ which arose from a face $f_0$ by adding new vertices in the first inner loop of Algorithm \ref{core_ga}. 
Since (i) holds after iteration $i$, applying the same partitioning rule in iteration $i+1$ yields 
$\bar V_+ \subseteq \hat V_+$ and $\bar V_- \subseteq \hat V_-$. Therefore, if a new vertex $v$ is added to $\bar V_0$ and $\bar V$ in the first inner loop, this vertex is also added to $\hat V_0$ and $\hat V$.
Thus, for every new vertex $v$ we have $I(v) \supseteq I(u) \cap I(w)$, where $u,w \in V(f)$, which yields that $N_f \geq d-2$ is maintained. In the remaining case, the face $f$ arose from deleting edges and vertices after the second inner loop. Then we have $V(f) \subseteq \bar V_0 \subseteq \hat V_0$. Thus $i+1 \in \bigcap\{I(v)\mid v \in V(f)\}$ and hence $N_f \geq 1 \geq d-2$. We have shown that (iii) and (i) hold after iteration $i+1$. 

Let $e=u_0w_0$ be an edge added in iteration $i+1$ to $\bar E$. Then $u_0,w_0$ belong to the same face $f$ of $\bar G$ and thus $|I(u_0) \cap I(w_0)| \geq d-2$ by (iii). We also have $u_0,w_0 \in \bar V_0 \subseteq \hat V_0$. Thus, the index $i+1$ is added to $I(u_0) \cap I(w_0)$ which yields $|I(u_0) \cap I(w_0)| \geq d-1$. Therefore, the edge $u_0w_0$ is also added to $\hat E$.

Since multiples of edges, which can be generated in the part after the second inner loop, are removed, $\bar G$ is a graph. It remains to show that $\bar G$ is planar. This follows by induction because $K_3$ and $K_4$ are planar and new edges are only inserted within faces.
\end{proof}

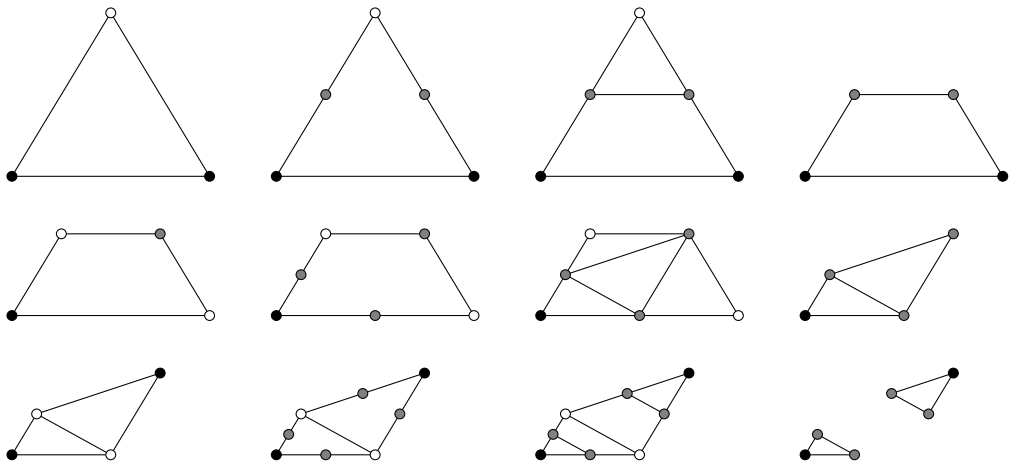
\begin{figure}[hpt]
	\def\scalefactor{1.3}
	\def\spacebetweenpictures{.5cm}
	\begin{center}
\begin{tikzpicture}[scale=\scalefactor]
	\coordinate (V1)  at (0,0);
	\coordinate (V2)  at (2,0);
	\coordinate (V3)  at (1,5/3);	
	\draw (V1) -- (V2) -- (V3) -- (V1);	
	\draw[fill=black] (V1) circle (0.05);
	\draw[fill=black]  (V2) circle (0.05);
	\draw[fill=white] (V3) circle (0.05);	
\end{tikzpicture}
\hspace{\spacebetweenpictures}
\begin{tikzpicture}[scale=\scalefactor]
	\coordinate (V1)  at (0,0);
	\coordinate (V2)  at (2,0);
	\coordinate (V3)  at (1,5/3);	
	\coordinate (V4)  at (1/2,5/6); 
	\coordinate (V5)  at (3/2,5/6); 
	\draw[] (V1) -- (V2) -- (V3) -- (V1);	
	\draw[fill=black] (V1) circle (0.05);
	\draw[fill=black]  (V2) circle (0.05);
	\draw[fill=white] (V3) circle (0.05);	
	\draw[fill=gray]  (V4) circle (0.05);
	\draw[fill=gray]  (V5) circle (0.05);
\end{tikzpicture}
\hspace{\spacebetweenpictures}
\begin{tikzpicture}[scale=\scalefactor]
	\coordinate (V1)  at (0,0);
	\coordinate (V2)  at (2,0);
	\coordinate (V3)  at (1,5/3);	
	\coordinate (V4)  at (1/2,5/6); 
	\coordinate (V5)  at (3/2,5/6); 
	\draw[] (V1) -- (V2) -- (V3) -- (V1);	
	\draw[] (V4) -- (V5);
	\draw[fill=black] (V1) circle (0.05);
	\draw[fill=black]  (V2) circle (0.05);
	\draw[fill=white] (V3) circle (0.05);	
	\draw[fill=gray]  (V4) circle (0.05);
	\draw[fill=gray]  (V5) circle (0.05);
\end{tikzpicture}
\hspace{\spacebetweenpictures}
\begin{tikzpicture}[scale=\scalefactor]
	\coordinate (V1)  at (0,0);
	\coordinate (V2)  at (2,0);
	\coordinate (V3)  at (1,5/3);	
	\coordinate (V4)  at (1/2,5/6); 
	\coordinate (V5)  at (3/2,5/6); 
	\draw[] (V1) -- (V2) -- (V5) -- (V4) -- (V1);	
	\draw[fill=black] (V1) circle (0.05);
	\draw[fill=black]  (V2) circle (0.05);
	\draw[white,fill=white] (V3) circle (0.05);	
	\draw[fill=gray]  (V4) circle (0.05);
	\draw[fill=gray]  (V5) circle (0.05);
\end{tikzpicture}

\vspace{-.5cm}
\begin{tikzpicture}[scale=\scalefactor]
	\coordinate (V1)  at (0,0);
	\coordinate (V2)  at (2,0);
	\coordinate (V3)  at (1,5/3);	
	\coordinate (V4)  at (1/2,5/6); 
	\coordinate (V5)  at (3/2,5/6); 
	
	\draw[] (V1) -- (V2) -- (V5) -- (V4) -- (V1);	
	\draw[fill=black] (V1) circle (0.05);
	\draw[fill=white]  (V2) circle (0.05);
	\draw[white,fill=white] (V3) circle (0.05);	
	\draw[fill=white]  (V4) circle (0.05);
	\draw[fill=gray]  (V5) circle (0.05);
\end{tikzpicture}
\hspace{\spacebetweenpictures}
\begin{tikzpicture}[scale=\scalefactor]
	\coordinate (V1)  at (0,0);
	\coordinate (V2)  at (2,0);
	\coordinate (V3)  at (1,5/3);	
	\coordinate (V4)  at (1/2,5/6); 
	\coordinate (V5)  at (3/2,5/6); 
	\coordinate (V6)  at (1/4,5/12); 
	\coordinate (V7)  at (1,0); 
	\draw[] (V1) -- (V2) -- (V5) -- (V4) -- (V1);	
	\draw[fill=black] (V1) circle (0.05);
	\draw[fill=white]  (V2) circle (0.05);
	\draw[white,fill=white] (V3) circle (0.05);	
	\draw[fill=white]  (V4) circle (0.05);
	\draw[fill=gray]  (V5) circle (0.05);
	\draw[fill=gray]  (V6) circle (0.05);
	\draw[fill=gray]  (V7) circle (0.05);
\end{tikzpicture}
\hspace{\spacebetweenpictures}
\begin{tikzpicture}[scale=\scalefactor]
	\coordinate (V1)  at (0,0);
	\coordinate (V2)  at (2,0);
	\coordinate (V3)  at (1,5/3);	
	\coordinate (V4)  at (1/2,5/6); 
	\coordinate (V5)  at (3/2,5/6); 
	\coordinate (V6)  at (1/4,5/12); 
	\coordinate (V7)  at (1,0); 
	\draw[] (V1) -- (V2) -- (V5) -- (V4) -- (V1);	
	\draw[] (V5) -- (V6) -- (V7) -- (V5);	
	\draw[fill=black] (V1) circle (0.05);
	\draw[fill=white]  (V2) circle (0.05);
	\draw[white,fill=white] (V3) circle (0.05);	
	\draw[fill=white]  (V4) circle (0.05);
	\draw[fill=gray]  (V5) circle (0.05);
	\draw[fill=gray]  (V6) circle (0.05);
	\draw[fill=gray]  (V7) circle (0.05);
\end{tikzpicture}
\hspace{\spacebetweenpictures}
\begin{tikzpicture}[scale=\scalefactor]
	\coordinate (V1)  at (0,0);
	\coordinate (V2)  at (2,0);
	\coordinate (V3)  at (1,5/3);	
	\coordinate (V4)  at (1/2,5/6); 
	\coordinate (V5)  at (3/2,5/6); 
	\coordinate (V6)  at (1/4,5/12); 
	\coordinate (V7)  at (1,0); 
	\draw[] (V1) -- (V7) -- (V5) -- (V6) -- (V1);	
	\draw[] (V7) -- (V6);	
	\draw[fill=black] (V1) circle (0.05);
	\draw[white,fill=white]  (V2) circle (0.05);
	\draw[white,fill=white] (V3) circle (0.05);	
	\draw[white,fill=white]  (V4) circle (0.05);
	\draw[fill=gray]  (V5) circle (0.05);
	\draw[fill=gray]  (V6) circle (0.05);
	\draw[fill=gray]  (V7) circle (0.05);
\end{tikzpicture}

\vspace{-.5cm}
\begin{tikzpicture}[scale=\scalefactor]
	\coordinate (V1)  at (0,0);
	\coordinate (V2)  at (2,0);
	\coordinate (V3)  at (1,5/3);	
	\coordinate (V4)  at (1/2,5/6); 
	\coordinate (V5)  at (3/2,5/6); 
	\coordinate (V6)  at (1/4,5/12); 
	\coordinate (V7)  at (1,0); 
	\draw[] (V1) -- (V7) -- (V5) -- (V6) -- (V1);	
	\draw[] (V7) -- (V6);	
	\draw[fill=black] (V1) circle (0.05);
	\draw[white,fill=white]  (V2) circle (0.05);
	\draw[white,fill=white] (V3) circle (0.05);	
	\draw[white,fill=white]  (V4) circle (0.05);
	\draw[fill=black]  (V5) circle (0.05);
	\draw[fill=white]  (V6) circle (0.05);
	\draw[fill=white]  (V7) circle (0.05);
\end{tikzpicture}
\hspace{\spacebetweenpictures}
\begin{tikzpicture}[scale=\scalefactor]
	\coordinate (V1)  at (0,0);
	\coordinate (V2)  at (2,0);
	\coordinate (V3)  at (1,5/3);	
	\coordinate (V4)  at (1/2,5/6); 
	\coordinate (V5)  at (3/2,5/6); 
	\coordinate (V6)  at (1/4,5/12); 
	\coordinate (V7)  at (1,0); 
	\coordinate (V8)  at (1/2,0); 
	\coordinate (V9)  at (1/8,5/24); 
	\coordinate (V10)  at (5/4,5/12); 
	\coordinate (V11)  at (7/8,15/24); 
	\draw[] (V1) -- (V7) -- (V5) -- (V6) -- (V1);	
	\draw[] (V7) -- (V6);	
	\draw[fill=black] (V1) circle (0.05);
	\draw[white,fill=white]  (V2) circle (0.05);
	\draw[white,fill=white] (V3) circle (0.05);	
	\draw[white,fill=white]  (V4) circle (0.05);
	\draw[fill=black]  (V5) circle (0.05);
	\draw[fill=white]  (V6) circle (0.05);
	\draw[fill=white]  (V7) circle (0.05);
	\draw[fill=gray]  (V8) circle (0.05);
	\draw[fill=gray]  (V9) circle (0.05);
	\draw[fill=gray]  (V10) circle (0.05);
	\draw[fill=gray]  (V11) circle (0.05);
\end{tikzpicture}
\hspace{\spacebetweenpictures}
\begin{tikzpicture}[scale=\scalefactor]
	\coordinate (V1)  at (0,0);
	\coordinate (V2)  at (2,0);
	\coordinate (V3)  at (1,5/3);	
	\coordinate (V4)  at (1/2,5/6); 
	\coordinate (V5)  at (3/2,5/6); 
	\coordinate (V6)  at (1/4,5/12); 
	\coordinate (V7)  at (1,0); 
	\coordinate (V8)  at (1/2,0); 
	\coordinate (V9)  at (1/8,5/24); 
	\coordinate (V10)  at (5/4,5/12); 
	\coordinate (V11)  at (7/8,15/24); 
	\draw[] (V1) -- (V7) -- (V5) -- (V6) -- (V1);	
	\draw[] (V7) -- (V6);	
	\draw[] (V8) -- (V9);
	\draw[] (V10) -- (V11);
	\draw[fill=black] (V1) circle (0.05);
	\draw[white,fill=white]  (V2) circle (0.05);
	\draw[white,fill=white] (V3) circle (0.05);	
	\draw[white,fill=white]  (V4) circle (0.05);
	\draw[fill=black]  (V5) circle (0.05);
	\draw[fill=white]  (V6) circle (0.05);
	\draw[fill=white]  (V7) circle (0.05);
	\draw[fill=gray]  (V8) circle (0.05);
	\draw[fill=gray]  (V9) circle (0.05);
	\draw[fill=gray]  (V10) circle (0.05);
	\draw[fill=gray]  (V11) circle (0.05);
\end{tikzpicture}
\hspace{\spacebetweenpictures}
\begin{tikzpicture}[scale=\scalefactor]
	\coordinate (V1)  at (0,0);
	\coordinate (V2)  at (2,0);
	\coordinate (V3)  at (1,5/3);	
	\coordinate (V4)  at (1/2,5/6); 
	\coordinate (V5)  at (3/2,5/6); 
	\coordinate (V6)  at (1/4,5/12); 
	\coordinate (V7)  at (1,0); 
	\coordinate (V8)  at (1/2,0); 
	\coordinate (V9)  at (1/8,5/24); 
	\coordinate (V10)  at (5/4,5/12); 
	\coordinate (V11)  at (7/8,15/24); 
	\draw[] (V1) -- (V8) -- (V9) --(V1);		
	\draw[] (V5) -- (V10) -- (V11) -- (V5);
	\draw[fill=black] (V1) circle (0.05);
	\draw[white,fill=white]  (V2) circle (0.05);
	\draw[white,fill=white] (V3) circle (0.05);	
	\draw[white,fill=white]  (V4) circle (0.05);
	\draw[fill=black]  (V5) circle (0.05);
	\draw[white,fill=white]  (V6) circle (0.05);
	\draw[white,fill=white]  (V7) circle (0.05);
	\draw[fill=gray]  (V8) circle (0.05);
	\draw[fill=gray]  (V9) circle (0.05);
	\draw[fill=gray]  (V10) circle (0.05);
	\draw[fill=gray]  (V11) circle (0.05);
\end{tikzpicture}

\end{center}
\caption{Example to illustrate Algorithm \ref{core_ga} for $d=2$. Rows correspond to outer iterations. The most left column shows the partitioning (chosen arbitrarily) of the set $\bar V$ in each outer iteration: White, gray and black vertices represent the sets $\bar V_+$, $\bar V_0$ and $\bar V_-$, respectively. The second and third column show the graph $\bar G$ after the first and second inner loop, respectively. In particular, the ``shortcuts'' are inserted in the third column. The last column shows the graph $\bar G$ at the end of an outer iteration. The outer face is invalid.}\label{ex_core_ga_2}
\end{figure}

\begin{figure}[hp]
	\def\scalefactor{1.4}
	\def\spacebetweenpictures{.9cm}
	\begin{center}
\begin{tikzpicture}[scale=\scalefactor]
	\coordinate (V1)  at (0,0);
	\coordinate (V2)  at (2,0);
	\coordinate (V3)  at (1,5/3);	
	\draw (V1) -- (V2) -- (V3) -- (V1);		
	\draw[fill=black] (V1) circle (0.05) node[below] {\tiny $2,3$};
	\draw[fill=black] (V2) circle (0.05) node[below] {\tiny $1,3$};
	\draw[fill=white] (V3) circle (0.05) node[above] {\tiny $1,2$};	
\end{tikzpicture}
\hspace{\spacebetweenpictures}
\begin{tikzpicture}[scale=\scalefactor]
	\coordinate (V1)  at (0,0);
	\coordinate (V2)  at (2,0);
	\coordinate (V3)  at (1,5/3);	
	\coordinate (V4)  at (1/2,5/6); 
	\coordinate (V5)  at (3/2,5/6); 
	\draw[] (V1) -- (V2) -- (V3) -- (V1);	
	\draw[fill=black] (V1) circle (0.05) node[below] {\tiny $2,3$};
	\draw[fill=black] (V2) circle (0.05) node[below] {\tiny $1,3$};
	\draw[fill=white] (V3) circle (0.05) node[above] {\tiny $1,2$};		
	\draw[fill=gray]  (V4) circle (0.05) node[left] {\tiny $2,4$};
	\draw[fill=gray]  (V5) circle (0.05) node[right] {\tiny $1,4$};
\end{tikzpicture}
\hspace{\spacebetweenpictures}
\begin{tikzpicture}[scale=\scalefactor]
	\coordinate (V1)  at (0,0);
	\coordinate (V2)  at (2,0);
	\coordinate (V3)  at (1,5/3);	
	\coordinate (V4)  at (1/2,5/6); 
	\coordinate (V5)  at (3/2,5/6); 
	\draw[] (V1) -- (V2) -- (V5) -- (V4) -- (V1);	
	\draw[fill=black] (V1) circle (0.05) node[below] {\tiny $2,3$};
	\draw[fill=black] (V2) circle (0.05) node[below] {\tiny $1,3$};
	\draw[white,fill=white] (V3) circle (0.05) node[above] {\tiny $1,2$};		
	\draw[fill=gray]  (V4) circle (0.05) node[left] {\tiny $2,4$};
	\draw[fill=gray]  (V5) circle (0.05) node[right] {\tiny $1,4$};
\end{tikzpicture}

\vspace{-1cm}
\begin{tikzpicture}[scale=\scalefactor]
	\coordinate (V1)  at (0,0);
	\coordinate (V2)  at (2,0);
	\coordinate (V3)  at (1,5/3);	
	\coordinate (V4)  at (1/2,5/6); 
	\coordinate (V5)  at (3/2,5/6); 
	\draw[] (V1) -- (V2) -- (V5) -- (V4) -- (V1);
	\draw[fill=black] (V1) circle (0.05) node[below] {\tiny $2,3$};
	\draw[fill=white] (V2) circle (0.05) node[below] {\tiny $1,3$};
	\draw[white,fill=white] (V3) circle (0.05) node[above] {\tiny $1,2$};		
	\draw[fill=white]  (V4) circle (0.05) node[left] {\tiny $2,4$};
	\draw[fill=gray]  (V5) circle (0.05) node[right] {\tiny $1,4$};		
	\draw[fill=black] (V1) circle (0.05);
	\draw[fill=white]  (V2) circle (0.05);
	\draw[white,fill=white] (V3) circle (0.05);	
	\draw[fill=white]  (V4) circle (0.05);
	\draw[fill=gray]  (V5) circle (0.05);
\end{tikzpicture}
\hspace{\spacebetweenpictures}
\begin{tikzpicture}[scale=\scalefactor]
	\coordinate (V1)  at (0,0);
	\coordinate (V2)  at (2,0);
	\coordinate (V3)  at (1,5/3);	
	\coordinate (V4)  at (1/2,5/6); 
	\coordinate (V5)  at (3/2,5/6); 
	\coordinate (V6)  at (1/4,5/12); 
	\coordinate (V7)  at (1,0); 
	\draw[] (V1) -- (V2) -- (V5) -- (V4) -- (V1);	
	\draw[fill=black] (V1) circle (0.05) node[below] {\tiny $2,3$};
	\draw[fill=white] (V2) circle (0.05) node[below] {\tiny $1,3$};
	\draw[white,fill=white] (V3) circle (0.05) node[above] {\tiny $1,2$};		
	\draw[fill=white]  (V4) circle (0.05) node[left] {\tiny $2,4$};
	\draw[fill=gray]  (V5) circle (0.05) node[right] {\tiny $1,4,5$};
	\draw[fill=gray]  (V6) circle (0.05) node[left] {\tiny $2,5$};
	\draw[fill=gray]  (V7) circle (0.05) node[below] {\tiny $3,5$};
\end{tikzpicture}
\hspace{\spacebetweenpictures}
\begin{tikzpicture}[scale=\scalefactor]
	\coordinate (V1)  at (0,0);
	\coordinate (V2)  at (2,0);
	\coordinate (V3)  at (1,5/3);	
	\coordinate (V4)  at (1/2,5/6); 
	\coordinate (V5)  at (3/2,5/6); 
	\coordinate (V6)  at (1/4,5/12); 
	\coordinate (V7)  at (1,0); 
	\draw[] (V1) -- (V7) -- (V5) -- (V6) -- (V1);	
	\draw[] (V7) -- (V6);	
	\draw[fill=black] (V1) circle (0.05) node[below] {\tiny $2,3$};
	\draw[white,fill=white] (V2) circle (0.05) node[below] {\tiny $1,3$};
	\draw[white,fill=white] (V3) circle (0.05) node[above] {\tiny $1,2$};		
	\draw[white,fill=white]  (V4) circle (0.05) node[left] {\tiny $2,4$};
	\draw[fill=gray]  (V5) circle (0.05) node[right] {\tiny $1,4,5$};
	\draw[fill=gray]  (V6) circle (0.05) node[left] {\tiny $2,5$};
	\draw[fill=gray]  (V7) circle (0.05) node[below] {\tiny $3,5$};
\end{tikzpicture}

\vspace{-1cm}
\begin{tikzpicture}[scale=\scalefactor]
	\coordinate (V1)  at (0,0);
	\coordinate (V2)  at (2,0);
	\coordinate (V3)  at (1,5/3);	
	\coordinate (V4)  at (1/2,5/6); 
	\coordinate (V5)  at (3/2,5/6); 
	\coordinate (V6)  at (1/4,5/12); 
	\coordinate (V7)  at (1,0); 
	\draw[] (V1) -- (V7) -- (V5) -- (V6) -- (V1);	
	\draw[] (V7) -- (V6);
	\draw[fill=black] (V1) circle (0.05) node[below left] {\tiny $2,3$};
	\draw[white,fill=white] (V2) circle (0.05) node[below] {\tiny $1,3$};
	\draw[white,fill=white] (V3) circle (0.05) node[above] {\tiny $1,2$};		
	\draw[white,fill=white]  (V4) circle (0.05) node[left] {\tiny $2,4$};
	\draw[fill=black]  (V5) circle (0.05) node[right] {\tiny $1,4,5$};
	\draw[fill=white]  (V6) circle (0.05) node[above left] {\tiny $2,5$};
	\draw[fill=white]  (V7) circle (0.05) node[below right] {\tiny $3,5$};
\end{tikzpicture}
\hspace{\spacebetweenpictures}
\begin{tikzpicture}[scale=\scalefactor]
	\coordinate (V1)  at (0,0);
	\coordinate (V2)  at (2,0);
	\coordinate (V3)  at (1,5/3);	
	\coordinate (V4)  at (1/2,5/6); 
	\coordinate (V5)  at (3/2,5/6); 
	\coordinate (V6)  at (1/4,5/12); 
	\coordinate (V7)  at (1,0); 
	\coordinate (V8)  at (1/2,0); 
	\coordinate (V9)  at (1/8,5/24); 
	\coordinate (V10)  at (5/4,5/12); 
	\coordinate (V11)  at (7/8,15/24); 
	\draw[] (V1) -- (V7) -- (V5) -- (V6) -- (V1);	
	\draw[] (V7) -- (V6);	
	\draw[fill=black] (V1) circle (0.05) node[below left] {\tiny $2,3$};
	\draw[white,fill=white] (V2) circle (0.05) node[below] {\tiny $1,3$};
	\draw[white,fill=white] (V3) circle (0.05) node[above] {\tiny $1,2$};		
	\draw[white,fill=white]  (V4) circle (0.05) node[left] {\tiny $2,4$};
	\draw[fill=black]  (V5) circle (0.05) node[right] {\tiny $1,4,5$};
	\draw[fill=white]  (V6) circle (0.05) node[above left] {\tiny $2,5$};
	\draw[fill=white]  (V7) circle (0.05) node[below right] {\tiny $3,5$};
	\draw[fill=gray]  (V8) circle (0.05) node[below] {\tiny $3,6$};
	\draw[fill=gray]  (V9) circle (0.05) node[left] {\tiny $2,6$};
	\draw[fill=gray]  (V10) circle (0.05) node[right] {\tiny $5,6$};
	\draw[fill=gray]  (V11) circle (0.05) node[above] {\tiny $5,6$};
\end{tikzpicture}
\hspace{\spacebetweenpictures}
\begin{tikzpicture}[scale=\scalefactor]
	\coordinate (V1)  at (0,0);
	\coordinate (V2)  at (2,0);
	\coordinate (V3)  at (1,5/3);	
	\coordinate (V4)  at (1/2,5/6); 
	\coordinate (V5)  at (3/2,5/6); 
	\coordinate (V6)  at (1/4,5/12); 
	\coordinate (V7)  at (1,0); 
	\coordinate (V8)  at (1/2,0); 
	\coordinate (V9)  at (1/8,5/24); 
	\coordinate (V10)  at (5/4,5/12); 
	\coordinate (V11)  at (7/8,15/24); 
	\draw[] (V1) -- (V8) -- (V9) --(V1);		
	\draw[] (V5) -- (V10) -- (V11) -- (V5);
	\draw[] (V8) -- (V10) -- (V9) -- (V11) --(V8);
	\draw[fill=black] (V1) circle (0.05) node[below left] {\tiny $2,3$};
	\draw[white,fill=white] (V2) circle (0.05) node[below] {\tiny $1,3$};
	\draw[white,fill=white] (V3) circle (0.05) node[above] {\tiny $1,2$};		
	\draw[white,fill=white]  (V4) circle (0.05) node[left] {\tiny $2,4$};
	\draw[fill=black]  (V5) circle (0.05) node[right] {\tiny $1,4,5$};
	\draw[white,fill=white]  (V6) circle (0.05) node[above left] {\tiny $2,5$};
	\draw[white,fill=white]  (V7) circle (0.05) node[below right] {\tiny $3,5$};
	\draw[fill=gray]  (V8) circle (0.05) node[below] {\tiny $3,6$};
	\draw[fill=gray]  (V9) circle (0.05) node[left] {\tiny $2,6$};
	\draw[fill=gray]  (V10) circle (0.05) node[right] {\tiny $5,6$};
	\draw[fill=gray]  (V11) circle (0.05) node[above] {\tiny $5,6$};
\end{tikzpicture}
\end{center}
\caption{Example to illustrate Algorithm \ref{core_addm} for $d=2$. We use the same symbols, the same example and the same partitioning rule as in Figure \ref{ex_core_ga_2}. Again, rows correspond to outer iterations and the left column shows the partitioning. The second column shows the graph $\hat G$ after the second inner loop and the last column shows the graph at the end of an outer iteration. The incidence list $I(v)$ is displayed at each vertex $v$. We see in Figure \ref{ex_core_ga_2} that Algorithm \ref{core_ga} has computed a (proper) subgraph of $\hat G$, compare Theorem \ref{th_rel_core}. }\label{ex_core_addm}
\end{figure}
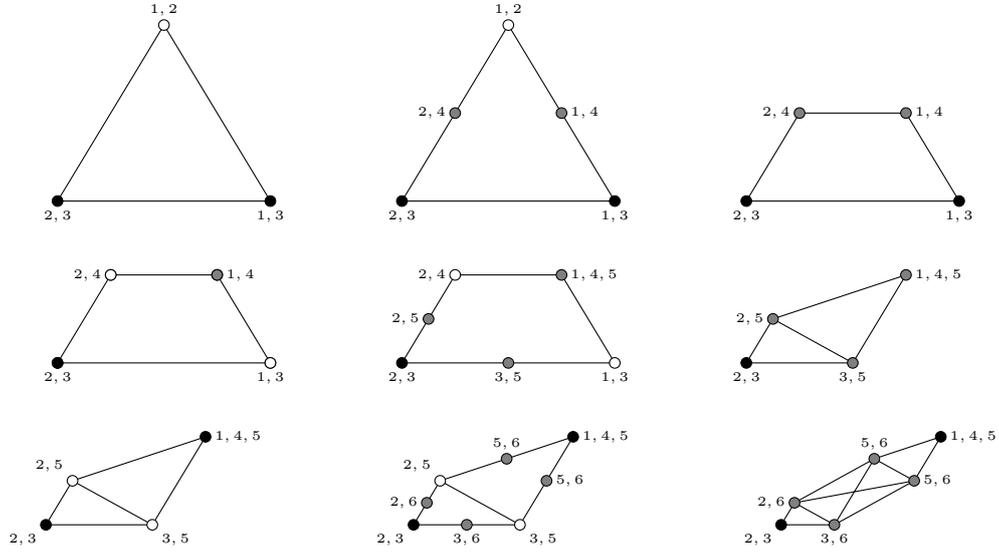

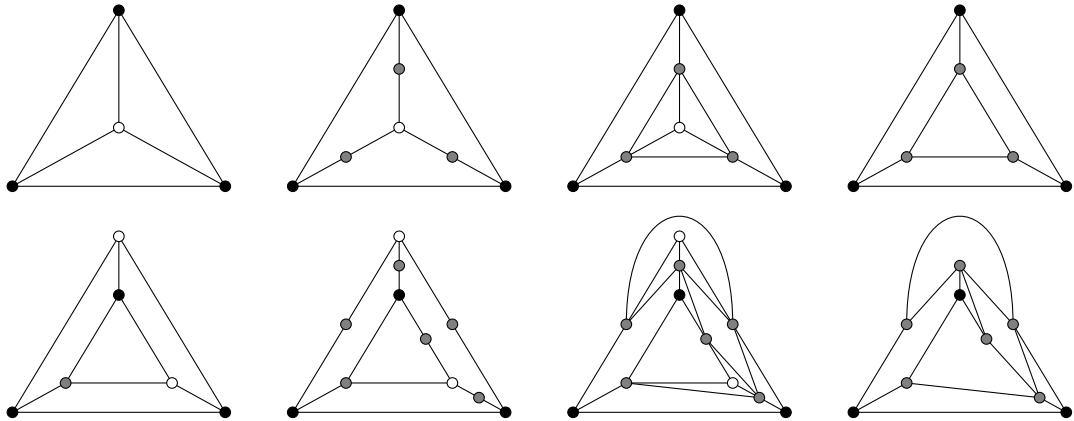
\begin{figure}[hpt]
\def\scalefactor{1.4}
\def\spacebetweenpictures{.5cm}
	\begin{center}
\begin{tikzpicture}[scale=\scalefactor]
	\coordinate (V1)  at (0,0);
	\coordinate (V2)  at (2,0);
	\coordinate (V3)  at (1,5/3);	
	\coordinate (V4)  at (1,5/9);	
	\draw (V1) -- (V2) -- (V3) -- (V1);	
	\draw (V1) -- (V4);
	\draw (V2) -- (V4);	
	\draw (V3) -- (V4);
	\draw[fill=black] (V1) circle (0.05);
	\draw[fill=black] (V2) circle (0.05);
	\draw[fill=black] (V3) circle (0.05);	
	\draw[fill=white] (V4) circle (0.05);
\end{tikzpicture}
\hspace{\spacebetweenpictures}
\begin{tikzpicture}[scale=\scalefactor]
	\coordinate (V1)  at (0,0);
	\coordinate (V2)  at (2,0);
	\coordinate (V3)  at (1,5/3);	
	\coordinate (V4)  at (1,5/9);
	\coordinate (V5)  at (1/2,5/18); 
	\coordinate (V6)  at (3/2,5/18); 
	\coordinate (V7)  at (1,10/9); 
	\draw (V1) -- (V2) -- (V3) -- (V1);	
	\draw (V1) -- (V4);
	\draw (V2) -- (V4);	
	\draw (V3) -- (V4);
	\draw[fill=black] (V1) circle (0.05);
	\draw[fill=black] (V2) circle (0.05);
	\draw[fill=black] (V3) circle (0.05);	
	\draw[fill=white] (V4) circle (0.05);
	\draw[fill=gray] (V5) circle (0.05);
	\draw[fill=gray] (V6) circle (0.05);	
	\draw[fill=gray] (V7) circle (0.05);
\end{tikzpicture}
\hspace{\spacebetweenpictures}
\begin{tikzpicture}[scale=\scalefactor]
	\coordinate (V1)  at (0,0);
	\coordinate (V2)  at (2,0);
	\coordinate (V3)  at (1,5/3);	
	\coordinate (V4)  at (1,5/9);
	\coordinate (V5)  at (1/2,5/18); 
	\coordinate (V6)  at (3/2,5/18); 
	\coordinate (V7)  at (1,10/9); 
	\draw (V1) -- (V2) -- (V3) -- (V1);	
	\draw (V5) -- (V6) -- (V7) -- (V5);
	\draw (V1) -- (V4);
	\draw (V2) -- (V4);	
	\draw (V3) -- (V4);
	\draw[fill=black] (V1) circle (0.05);
	\draw[fill=black] (V2) circle (0.05);
	\draw[fill=black] (V3) circle (0.05);	
	\draw[fill=white] (V4) circle (0.05);
	\draw[fill=gray] (V5) circle (0.05);
	\draw[fill=gray] (V6) circle (0.05);	
	\draw[fill=gray] (V7) circle (0.05);
\end{tikzpicture}
\hspace{\spacebetweenpictures}
\begin{tikzpicture}[scale=\scalefactor]
	\coordinate (V1)  at (0,0);
	\coordinate (V2)  at (2,0);
	\coordinate (V3)  at (1,5/3);	
	\coordinate (V4)  at (1,5/9);
	\coordinate (V5)  at (1/2,5/18); 
	\coordinate (V6)  at (3/2,5/18); 
	\coordinate (V7)  at (1,10/9); 
	\draw (V1) -- (V2) -- (V3) -- (V1);	
	\draw (V5) -- (V6) -- (V7) -- (V5);
	\draw (V1) -- (V5);
	\draw (V2) -- (V6);	
	\draw (V3) -- (V7);
	\draw[fill=black] (V1) circle (0.05);
	\draw[fill=black] (V2) circle (0.05);
	\draw[fill=black] (V3) circle (0.05);	
	\draw[white,fill=white] (V4) circle (0.05);
	\draw[fill=gray] (V5) circle (0.05);
	\draw[fill=gray] (V6) circle (0.05);	
	\draw[fill=gray] (V7) circle (0.05);
\end{tikzpicture}

\vspace{-.2cm}
\begin{tikzpicture}[scale=\scalefactor]
	\coordinate (V1)  at (0,0);
	\coordinate (V2)  at (2,0);
	\coordinate (V3)  at (1,5/3);	
	\coordinate (V4)  at (1,5/9);
	\coordinate (V5)  at (1/2,5/18); 
	\coordinate (V6)  at (3/2,5/18); 
	\coordinate (V7)  at (1,10/9); 
	\draw (V1) -- (V2) -- (V3) -- (V1);	
	\draw (V5) -- (V6) -- (V7) -- (V5);
	\draw (V1) -- (V5);
	\draw (V2) -- (V6);	
	\draw (V3) -- (V7);
	\draw[fill=black] (V1) circle (0.05);
	\draw[fill=black] (V2) circle (0.05);
	\draw[fill=white] (V3) circle (0.05);	
	\draw[white,fill=white] (V4) circle (0.05);
	\draw[fill=gray] (V5) circle (0.05);
	\draw[fill=white] (V6) circle (0.05);	
	\draw[fill=black] (V7) circle (0.05);
\end{tikzpicture}
\hspace{\spacebetweenpictures}
\begin{tikzpicture}[scale=\scalefactor]
	\coordinate (V1)  at (0,0);
	\coordinate (V2)  at (2,0);
	\coordinate (V3)  at (1,5/3);	
	\coordinate (V4)  at (1,5/9);
	\coordinate (V5)  at (1/2,5/18); 
	\coordinate (V6)  at (3/2,5/18); 
	\coordinate (V7)  at (1,10/9); 
	\coordinate (V8)  at (3/2,5/6); 
	\coordinate (V9)  at (1/2,5/6); 
	\coordinate (V10)  at (1,25/18); 
	\coordinate (V11)  at (7/4,5/36); 
	\coordinate (V12)  at (5/4,25/36); 
	\draw (V1) -- (V2) -- (V3) -- (V1);	
	\draw (V5) -- (V6) -- (V7) -- (V5);
	\draw (V1) -- (V5);
	\draw (V2) -- (V6);	
	\draw (V3) -- (V7);
	\draw[fill=black] (V1) circle (0.05);
	\draw[fill=black] (V2) circle (0.05);
	\draw[fill=white] (V3) circle (0.05);	
	\draw[white,fill=white] (V4) circle (0.05);
	\draw[fill=gray] (V5) circle (0.05);
	\draw[fill=white] (V6) circle (0.05);	
	\draw[fill=black] (V7) circle (0.05);	
	\draw[fill=gray] (V8) circle (0.05);
	\draw[fill=gray] (V9) circle (0.05);
	\draw[fill=gray] (V10) circle (0.05);
	\draw[fill=gray] (V11) circle (0.05);
	\draw[fill=gray] (V12) circle (0.05);
\end{tikzpicture}
\hspace{\spacebetweenpictures}
\begin{tikzpicture}[scale=\scalefactor]
	\coordinate (V1)  at (0,0);
	\coordinate (V2)  at (2,0);
	\coordinate (V3)  at (1,5/3);	
	\coordinate (V4)  at (1,5/9);
	\coordinate (V5)  at (1/2,5/18); 
	\coordinate (V6)  at (3/2,5/18); 
	\coordinate (V7)  at (1,10/9); 
	\coordinate (V8)  at (3/2,5/6); 
	\coordinate (V9)  at (1/2,5/6); 
	\coordinate (V10)  at (1,25/18); 
	\coordinate (V11)  at (7/4,5/36); 
	\coordinate (V12)  at (5/4,25/36); 
	\draw (V1) -- (V2);
	\draw (V1) -- (V5);
	\draw (V1) -- (V9);
		
	\draw (V2) -- (V8);		
	\draw (V2) -- (V11);	
	
	\draw (V3) -- (V8);
	\draw (V3) -- (V9);
	\draw (V3) -- (V10);
	
	\draw (V5) -- (V7);
	\draw (V5) -- (V6);
	\draw (V5) -- (V9);
	\draw (V5) -- (V10);
	\draw (V5) -- (V11);
	\draw (V5) -- (V12);
	
	\draw (V6) -- (V11);
	\draw (V6) -- (V12);	
		
	\draw (V7) -- (V10);
	\draw (V7) -- (V12);
	
	\draw (V8) to [out=90,in=90,looseness=3.5] (V9);
	\draw (V8) -- (V10);
	\draw (V8) -- (V11);

	\draw (V9) -- (V10);
	
	\draw (V10) -- (V12);

	\draw (V11) -- (V12);
	
	\draw[fill=black] (V1) circle (0.05);
	\draw[fill=black] (V2) circle (0.05);
	\draw[fill=white] (V3) circle (0.05);	
	\draw[fill=gray] (V5) circle (0.05);
	\draw[fill=white] (V6) circle (0.05);	
	\draw[fill=black] (V7) circle (0.05);	
	\draw[fill=gray] (V8) circle (0.05);
	\draw[fill=gray] (V9) circle (0.05);
	\draw[fill=gray] (V10) circle (0.05);
	\draw[fill=gray] (V11) circle (0.05);
	\draw[fill=gray] (V12) circle (0.05);	
\end{tikzpicture}
\hspace{\spacebetweenpictures}
\begin{tikzpicture}[scale=\scalefactor]
	\coordinate (V1)  at (0,0);
	\coordinate (V2)  at (2,0);
	\coordinate (V3)  at (1,5/3);	
	\coordinate (V4)  at (1,5/9);
	\coordinate (V5)  at (1/2,5/18); 
	\coordinate (V6)  at (3/2,5/18); 
	\coordinate (V7)  at (1,10/9); 
	\coordinate (V8)  at (3/2,5/6); 
	\coordinate (V9)  at (1/2,5/6); 
	\coordinate (V10)  at (1,25/18); 
	\coordinate (V11)  at (7/4,5/36); 
	\coordinate (V12)  at (5/4,25/36); 
	\draw (V1) -- (V2);
	\draw (V1) -- (V5);
	\draw (V1) -- (V9);
		
	\draw (V2) -- (V8);		
	\draw (V2) -- (V11);	
	
	\draw (V5) -- (V7);
	\draw (V5) -- (V9);
	\draw (V5) -- (V10);
	\draw (V5) -- (V11);
	\draw (V5) -- (V12);
		
	\draw (V7) -- (V10);
	\draw (V7) -- (V12);
	
	\draw (V8) to [out=90,in=90,looseness=3.5] (V9);
	\draw (V8) -- (V10);
	\draw (V8) -- (V11);

	\draw (V9) -- (V10);
	
	\draw (V10) -- (V12);

	\draw (V11) -- (V12);
	
	\draw[fill=black] (V1) circle (0.05);
	\draw[fill=black] (V2) circle (0.05);
	\draw[white,fill=white] (V3) circle (0.05);	
	\draw[fill=gray] (V5) circle (0.05);
	\draw[white,fill=white] (V6) circle (0.05);	
	\draw[fill=black] (V7) circle (0.05);	
	\draw[fill=gray] (V8) circle (0.05);
	\draw[fill=gray] (V9) circle (0.05);
	\draw[fill=gray] (V10) circle (0.05);
	\draw[fill=gray] (V11) circle (0.05);
	\draw[fill=gray] (V12) circle (0.05);	
\end{tikzpicture}
\end{center}
\caption{Example to illustrate Algorithm \ref{core_ga} for $d=3$. Rows correspond to outer iterations. The most left column shows the partitioning (chosen arbitrarily) of the set $\bar V$: White, gray and black vertices represent the sets $V_+$, $V_0$ and $V_-$, respectively. The second and the third column show the results of the first and second inner loop, respectively. In particular, the ``shortcuts'' are inserted in the third column. The last column shows the graph at the end of an outer iteration. In contrast to Figure \ref{ex_core_ga_2} all faces are valid.}\label{ex_core_ga_3}
\end{figure}

\begin{algorithm}[ph]
\DontPrintSemicolon
\SetKwInOut{Input}{input}\SetKwInOut{Output}{output}
\Input{$A \in \R^{m\times d}$, $d \in \{2,3\}$ such that $P \colonequals  \{x \mid A x \leq \one\}$ is a polytope and $S \colonequals  \{x \mid A_{[d+1]} x \leq \one\}$ is a simplex; tolerance $\varepsilon \geq 0$}
\Output{some plane graph $G=(V,E)$ and a coordinate function $c: V \to \R^d$ such that $P \subseteq \conv c(V) \subseteq (1+\varepsilon) P$}
\Begin{
	initialize the (plane) graph $G = G(V,E)$ by $K_{d+1}$ (taking into account (I1) and (I2)) with nodes $V=\{v_1,\dots,v_{d+1}\}$\;
	compute vertices $\{u_1,\dots,u_{d+1}\}$ of $(1+\frac{\varepsilon}{2})S$ and set $c(v_i) \leftarrow u_i$ for $i \in \{1,\dots,d+1\}$\;

    \For{$i\leftarrow d+2$ {\textbf{to}} $m$}{
		partition $V$ into disjoint sets $V_- \neq \emptyset$, $V_0$, $V_+$ such that
		$V_- \subseteq \{v \in V \mid A_i c(v) < 1 + \frac{\varepsilon}{2} \}$,
		$V_+ \subseteq \{v \in V \mid A_i c(v) > 1 + \frac{\varepsilon}{2} \}$,
		$V_0 \subseteq \{v \in V \mid 1 \leq A_i c(v) \leq 1 + \varepsilon \}$\;
		\For{$e = uw \in E$ satisfying $u \in V_-$ and $w \in V_+$}{
			add a new node $v=v(u,w)$ to $V_0$ and $V$\;			
			in $E$ replace $uw$ by two new edges $uv$ and $vw$\;
			define $c(v)$ by choosing a point which is on the line segment between $c(u)$ and $c(w)$ such that $1 \leq A_i c(v) \leq 1 + \varepsilon$\;			
		}
		\For{$u_0 u_+, w_0 w_+ \in E$ with $u_0,w_0 \in V_0$, $u_0 \neq w_0$, $u_0 w_0 \not\in E$, $u_+,w_+ \in V_+$}{
			\If{there is a walk along a bounding walk of some valid face $f$ from $u_0$ to $w_0$ with all intermediate nodes belonging to $V_+$ or all intermediate nodes belonging to $V_-$}
			{
				$E \leftarrow E \cup \{u_0w_0\}$\;
			} 
		}
		remove all edges incident with some $v \in V_+$ from $E$\;
		remove $V_+$ from $V$\;
		remove multiples of edges from $E$		  
    }
}
\caption{Shortcut algorithm for approximate vertex enumeration}
\label{ga}
\end{algorithm}

Algorithm \ref{ga} is called the {\em shortcut algorithm}. It arises from Algorithm \ref{core_ga} by specifying a partitioning rule. The partitioning rule is defined by the approximate vertex enumeration problem instance. A coordinate function $c: V \to \R^d$ is iteratively defined and used for the partitioning rule. Note that the approximate double description method in Algorithm \ref{addm} arises from Algorithm \ref{core_addm} by the same partitioning rule. Therefore, Algorithm \ref{ga} computes a subgraph of the graph computed by Algorithm \ref{addm}, as stated by Theorem \ref{th_rel_core}. Moreover, at termination of Algorithms \ref{ga} and \ref{addm}, vertices $v \in  \bar V \subseteq \hat V$ have the same coordinates in both algorithms. This proves the following corollary.

\begin{corollary}\label{cor_rel}
	If Algorithm \ref{ga} is correct then Algorithm \ref{addm} is correct.
\end{corollary} 

It remains to prove correctness of Algorithm \ref{ga}.
Let $G=(V,E)$ be a plane graph and let $c:V \to \R^d$, $d \in \{2,3\}$ be a {\em coordinate function}.
By defining $c(V)\colonequals \{c(v)\mid v \in V\}$ we obtain a new vertex set in $\R^d$. For every $e \in E$, we define by $c(e) \colonequals \conv\{c(u),c(v)\}$ an associated line segment in $\R^d$.
The set $c(E)\colonequals \{c(e)\mid e \in E\}$ of such line segments defines a new set of edges in $\R^d$. 
We obtain an embedding of $G$ into $\R^d$ which is denoted by $c(G) = (c(V),c(E))$. 
Note that, even for $d=2$, $c(G)$ is not necessarily a plane graph (in the sense of an embedding into the plane). 
Note further that distinct vertices $u,v$ in $V$ can have the same coordinates $c(u),c(v)$ and thus line segments $c(e)$ can reduce to points. 

Let us first consider the case $d=2$. 
Let $r \in \R^2\setminus \{0\}$ be a direction.
Subsequently, we assume that for all $v \in V$, $c(v) \not \in \R \cdot r$ holds.
Let $f$ be a face of $G$ and let $W(f)$ be an associated bounding walk. 
Let $N(f,r)$ be the number of line segments $c(e)$ for edges $e$ in $W(f)$ which have common points with the ray $R \colonequals \R_+\cdot r$. 
For short we say that $N(f,r)$ is the {\em number of boundary crossings} of the polygon $c(W(f))$ with the ray $R$. If an edge $e$ is twice contained in $W$, it is counted twice in case of crossing. 

A plane graph $G=(V,E)$ together with a coordinate function $c : V \to \R^2$, $(G,c)$ for short, is called {\em regular} if  
\begin{enumerate}[(A{$_2$})]
	\item \label{ass_A2} For every $e \in E$ there exists an index $\kappa(e) \in \{1,\dots,m\}$ such that the line segment $c(e)$ belongs to the half-plane 
		\begin{equation*}
	  		B_+(e) \colonequals \{x \mid A_{\kappa(e)} x \geq 1\}.
	  \end{equation*}
	\item \label{ass_B2} For a dense subset of directions $r$ in $\R^2\setminus\{0\}$, the number $N(r) \colonequals \sum_{f \in F} N(f,r)$, where $F$ denotes the set of all valid faces of $G$, is odd.
\end{enumerate}

\begin{theorem} \label{th_ga_2}
	Algorithm \ref{ga} is correct for $d=2$.
\end{theorem}
\begin{proof}
	First, we show by induction that $(G,c)$ is regular at termination of Algorithm \ref{ga}. 
	
After initialization (i.e.\ directly before the outer loop), $G$ is the complete graph $K_3$ and $c(G)$ is a triangle with zero in its interior defined by three inequalities $A_{[3]} x \leq (1+\frac{\varepsilon}{2}) \one$. 
There is exactly one valid face $f$ with a bounding walk $W(f) = (v_1,e_3,v_2,e_1,v_3,e_2,v_1)$. 
In every vertex $c(v_i)$, $i=1,2,3$, two of the three inequalities are active (i.e.\ hold with equality). 
This implies property (\ref{ass_A2}$_2$). 
For a dense subset of directions $r$ we have $c(v_i) \not \in \R \cdot r$, $i=1,2,3$ and thus $N(f,r)=1$. 
Since $f$ is the only valid face, property (\ref{ass_B2}$_2$) follows. 
Thus $(G,c)$ is regular after initialization. 
	
	Assume now that $(G,c)$ is regular after iteration $i$. An edge which is added in iteration $i+1$ satisfies property (\ref{ass_A2}$_2$) after iteration $i+1$: either for $\kappa(uv) \colonequals \kappa(uw)$, $\kappa(vw) \colonequals \kappa(uw)$ if $uv$ and $vw$ arose from $uw$ in the first inner loop of $\kappa(e) \colonequals i+1$ if $e$ was added in the second inner loop. So let us consider property (\ref{ass_B2}$_2$). An edge $e=uv$ added to $E$ in iteration $i+1$ splits a valid face into two new valid faces. If $c(e)$ crosses the ray $R$, $N(r)$ is increased by $2$, otherwise it is not changed. Thus $N(r)$ remains odd. If an edge $e \in E$ with $c(e)$ crossing $R$ is deleted in iteration $i+1$, we distinguish four cases. (i) If $e$ is incident with two valid faces, $N(r)$ is decreased by $2$. (ii) If $e$ is incident with exactly one valid face, it occurs twice in a bounding walk of this face and thus $N(r)$ is decreased by $2$. (iii) If $e$ is incident with exactly one invalid face, deleting it has no influence to $N(r)$ because only valid faces are involved. (iv) If $e$ is incident with a valid face $f_1$ and an invalid face $f_2$, these faces are merged into an invalid face $f$. For an edge $e=uv$ to be deleted we have $u,v \in V_0 \cup V_+$ since $u$ (or $v$) belongs to $V_+$ and by the first inner loop of the algorithm, $v$ (or $u$) cannot belong to $V_-$. Moreover, by the second inner loop of Algorithm \ref{ga}, all vertices of the bounding walk $W(f_1)$ belong to $V_0 \cup V_+$. This implies that for every vertex $v$ in $W(f_1)$, $c(v)$ belongs to the half-plane $\{x \in \R^2 \mid  A_{i+1} x \geq 1\}$. Thus $N(f_1,r)$ coincides with the number of line segments $c(e)$ for edges $e$ in $W(f_1)$ which have common points with the line $\R\cdot r$. The polygon in $\R^2$ defined by $W(f_1)$ and $c$ has an even number $k$ of crossing points with the line $\R\cdot r$. Since the new face $f$ is invalid, deleting $e$ reduces $N(r)$ exactly by this even number $k$. This shows that property (\ref{ass_B2}$_2$) is maintained.
	
Secondly, we show the inclusions $P \subseteq \conv c(V) \subseteq (1+\varepsilon) P$. The second inclusion is obviously satisfied, since $V_+$ is deleted at the end of each outer iteration. So let us prove the first inclusion. At termination of the algorithm, let $r \in P \setminus \{0\}$ such that (\ref{ass_B2}$_2$) holds. Since $N(r)$ is odd, there exists a face $f$ such that $N(f,r) > 0$. Hence there exists an edge $e \in E$ and some $\mu > 0$ such that $\mu\cdot r \in c(e)$. By (\ref{ass_A2}$_2$) we have $c(e) \in B_+(e)$ and thus $c(e) \cap \inter P = \emptyset$. We obtain $\mu \geq 1$. We conclude: For a dense subset of directions $r \in P \setminus \{0\}$ there exists $\mu \geq 1$ with $\mu r \in \conv c(V)$. Since $P$ and $\conv c(V)$ are convex polytopes, we obtain the inclusion $P \subseteq \conv c(V)$. 
\end{proof}

Now let us consider the case $d=3$.
Let $r \in \R^3\setminus \{0\}$ be a direction.
For a vector $a \in \R^3$, linearly independent of $r$, the set 
$$M \colonequals M(r,a) \colonequals \R \cdot r + \R_+ \cdot a$$
defines a half-plane in $\R^3$ whose relative boundary is the line $\R\cdot r$. Subsequently, we assume that for all $v \in V$, $c(v) \not \in M$ and for all edges $e \in E$, $c(e) \cap \R \cdot r = \emptyset$.

Let $f$ be a face of $G$ and let $W(f)$ be an associated bounding walk. Let
$N(f,r,a)$ be the number of line segments $c(e)$ for edges $e$ in $W(f)$ which have a common point with the half-plane $M=M(r,a)$. 
For short we say that $N(f,r,a)$ is the {\em number of boundary crossings} of the skew polygon $c(W(f))$ with the half-plane $M$. If an edge $e$ is twice contained in $W$, it is counted twice in case of crossing. We start with a statement which is used below in the proof of the correctness result.

\begin{proposition} \label{prop_ind}
	The property of $N(f,r,a)$ being odd is independent of the choice of the vector $a$.
\end{proposition}
\begin{proof}
Let $a, a' \in \R^3$ be two different choices with corresponding sets $M=M(r,a)$ and $M'=M(r,a')$ and corresponding numbers $N(f,r,a)$ and $N(f,r,a')$. If $M = M'$ the statement is obvious. Otherwise the set $M \cup M'$ partitions the space $\R^3$ into two regions.
We assumed that for no edge $e$ of the bounding walk $W(f)$, $c(e)$ crosses the line $M \cap M' = \R \cdot r$ and for no vertex $v$ of the bounding walk $W(f)$, $c(v)$ belongs to $M \cup M'$. Since a skew polygon defined by $W(f)$ and $c$ is a closed curve in $\R^3$, the number $m+m'$ of edges crossing the set $M \cup M'$ is even. Thus if $m$ is odd then so is $m'$.
\end{proof}

A plane graph $G=(V,E)$ together with a coordinate function $c : V \to \R^3$, $(G,c)$ for short, is called {\em regular} if  
\begin{enumerate}[(A{$_3$})]
	\item \label{ass_A3} For every face $f$ of $G$ there exists an index $\kappa(f) \in \{1,\dots,m\}$ such that for every edge $e$ of a bounding walk $W(f)$, the line segment $c(e)$ belongs to the half-plane 
		\begin{equation*}
	  		B_+(f) \colonequals \{x \mid A_{\kappa(f)} x \geq 1\}.
	  \end{equation*}
	\item \label{ass_B3} For a dense subset of the set of two linearly independent vectors $r,a \in \R^3$, the number $N(r,a) \colonequals \sum_{f \in F(r)} N(f,r,a)$, where $F(r)$ denotes the set of all faces of $G$ satisfying $A_{\kappa(f)} r > 0$, is odd.
\end{enumerate}

\begin{theorem} \label{th_ga_3}
	Algorithm \ref{ga} is correct for $d=3$.
\end{theorem}
\begin{proof}
	First, we show by induction that $(G,c)$ is regular at termination of Algorithm \ref{ga}. 
	
After initialization (i.e.\ directly before the outer loop), $G$ is the complete graph $K_4$ and $c(G)$ is a simplex with zero in its interior defined by four inequalities $A_{[4]} x \leq (1+\frac{\varepsilon}{2}) \one$. In every vertex $c(v_i)$, $i=1,\dots,4$, three of the four inequalities are active (i.e.\ hold with equality). This implies property (\ref{ass_A3}$_3$). 

For a dense subset of the set of two linearly independent vectors $r,a \in \R^3$ we have $c(v) \not \in M(r,a)$ and $c(e) \cap \R_+\cdot r =\emptyset$ for any vertex $v$ and any edge $e$ which occur in the algorithm. Let such vectors $r$ and $a$ be fixed. There is exactly one face $f_+$ of $K_4$ such that $\R_+ \cdot r$ crosses the relative interior of the facet $c(W(f_+))$ of the simplex $c(K_4)$. And there is exactly one face $f_- \neq f_+$ of $K_4$ such that $\R_- \cdot r$ crosses the relative interior of the facet $c(W(f_-))$ of $c(K_4)$. We have $N(f_+,r,a)=N(f_-,r,a)=1$. Clearly, $f_+$ belongs to $F(r)$ and $f_-$ does not. For the remaining two faces $f$ of $K_4$, $N(f,r,a)$ is even. Consequently, $N(r,a)$ is odd and hence $(G,c)$ is regular after initialization. 

	Assume now that $(G,c)$ is regular after iteration $i$. Let $f$ be a face of $G$ which is created in iteration $i+1$. Adding new vertices $v$ in the first inner loop, maintains property (\ref{ass_A3}$_3$) since $c(v)$ is chosen on the line between $c(u),c(w)$ for $u,w$ belonging to the same face.
	If $f$ is created in the second inner loop of iteration $i+1$ by adding an edge then another face $f_0$ were split into two new faces $f$ and $f_1$. Property (\ref{ass_A3}$_3$) holds for the parent face $f_0$ and consequently also for the new faces $f$ and $f_1$. In the last case, $f$ is created by merging two faces $f_1$ and $f_2$ into a new face $f$ as a consequence of deleting an edge $e=uv$ after the second inner loop. In this situation, we have $u,v \in V_0 \cup V_+$. Moreover, by the second inner loop, all vertices of the bounding walks $W(f_1)$ and $W(f_2)$ belong to $V_0 \cup V_+$. Hence, the 
	vertices of the bounding walk $W(f)$ belong to $V_0 \cup V_+$, which yields that for every vertex $v$ in $W(f)$, $c(v)$ belongs to the half-space $\{x \in \R^2 \mid  A_{i+1} x \geq 1\}$. This means that (\ref{ass_A3}$_3$) holds for such faces with $\kappa(f) = i+1$. Consequently, (\ref{ass_A3}$_3$) is maintained in iteration $i+1$.
	
	So let us consider property (\ref{ass_B3}$_3$). An edge $e=uv$ added to $E$ in iteration $i+1$ splits a face $f$ into two new faces. If $c(e)$ crosses the half-plane $M$ and $f\in F(r)$, $N(r,a)$ is increased by $2$, otherwise it is not changed. Thus $N(r,a)$ remains odd.

	If an edge $e \in E$ with $c(e)$ crossing $M$ is deleted in iteration $i+1$, we distinguish two cases. (i) Let $e$ be incident with exactly one face $f$. As already seen, we have $\kappa(f) = i+1$. If $A_{i+1} r > 0$, $N(r,a)$ is decreased by $2$, otherwise it is not changed. (ii) Let $e$ be incident with exactly two faces $f_1$ and $f_2$ which are merged to a new face $f$ by deleting $e$. Again, we have $\kappa(f_1)=\kappa(f_2)=i+1$. If $A_{i+1} r > 0$, both $f_1$ and $f_2$ belong to $F(r)$, otherwise both do not. Thus $N(r,a)$ is either decreased by $2$ or is left unchanged.
This shows that property (\ref{ass_B3}$_3$) is maintained.
	
Secondly, we show the inclusions $P \subseteq \conv c(V) \subseteq (1+\varepsilon) P$. The second inclusion is obviously satisfied, since $V_+$ is deleted at the end of each outer iteration. Let us prove the first inclusion. At termination of the algorithm, let $r \in P \setminus \{0\}$ and $a \in \R^3$, linearly independent of $r$ such that (\ref{ass_B3}$_3$) holds. Since $N(r,a)$ is odd, there exists a face $f$ such that $N(f,r,a)$ is odd. By Proposition \ref{prop_ind}, $N(f,r,a)$ is odd for a dense subset of all possible choices of $a$. Thus, without loss of generality, $N(f,r,-a)$ is odd. This implies that
$N(f,r,a)>0$ and $N(f,r,-a)>0$. Hence there are edges $e_+,e_-$ in $W(f)$ and crossing points $x_+ \in c(e_+) \cap M(r,a)$ and $x_- \in c(e_-) \cap M(r,-a)$. We have $H \colonequals M(r,-a) = -M(r,a)$ and thus $M(r,a)\cup M(r,-a)$ is a plane. Every convex combination $x$ of $x_+$ and $x_-$ belongs to $H$ and there is one such convex combination $x$ which belongs to the line $\R\cdot r$. By property (\ref{ass_A3}$_3$) we have $A_{\kappa(f)} x_+ \geq 1$ and $A_{\kappa(f)} x_- \geq 1$ and thus $A_{\kappa(f)} x \geq 1$. Since $A_{\kappa(f)} r > 0$, there exists $\mu > 0$ such that $x = \mu r$. By $A_{\kappa(f)} \mu r \geq 1$ and $A_{\kappa(f)} r \leq 1$, we conclude $\mu \geq 1$. Since $x_+$ and $x_-$ are convex combinations of $c(v)$ for vertices $v$ of $W(f)$, $x$ has the same property and thus $x=\mu r$ belongs to $\conv c(V)$. We conclude: For a dense subset of directions $r \in P \setminus \{0\}$ there exists $\mu \geq 1$ with $\mu r \in \conv c(V)$. Since $P$ and $\conv c(V)$ are convex polytopes, we obtain the inclusion $P \subseteq \conv c(V)$. 
\end{proof}

\section{Using imprecise arithmetic}\label{sec_imprec}

In this section we show that specific variants of both Algorithms \ref{addm} and \ref{ga} from the previous section (i.e.\ $d\leq 3$) are still correct if imprecise arithmetic (such as floating point arithmetic) is used and the imprecision of computations is not too high. 

The specifications (which are the same in both algorithms) are as follows:
\begin{enumerate}[(i)]
	\item For a vertex $v$ added to $V$ in the outer iteration $i$, define $c(v)$ such that $A_i c(v) = 1 + \frac{\varepsilon}{2}$ (which specifies the condition $1 \leq A_i c(v) \leq  1 + \varepsilon$ in both algorithms)
	\item Partition $V$ into sets $V_-$, $V_0$, $V_+$ in the outer iteration $i$ by setting 
	\begin{enumerate}[(a)]
		\item $V_- \colonequals \{v \in V\mid A_i c(v) < 1 +\frac{\varepsilon}{4}\} \subseteq \{v \in V\mid A_i c(v) < 1 +\frac{\varepsilon}{2}\}$
		\item $V_0 \colonequals \{v \in V\mid 1+ \frac{\varepsilon}{4} \leq A_i c(v) \leq 1 +\frac{3}{4}\varepsilon\}\subseteq \{v \in V\mid 1 \leq A_i c(v) \leq 1 + \varepsilon\}$,
		\item $V_+ \colonequals \{v \in V\mid A_i c(v) > 1 +\frac{3}{4}\varepsilon\} \subseteq \{v \in V\mid A_i c(v) > 1 +\frac{\varepsilon}{2}\}$
	\end{enumerate}
\end{enumerate}
Clearly, Algorithms \ref{addm} and \ref{ga} are still correct with these specifications.

The idea now is to formulate a condition which guaranties that the computation of $c(v)$ with imprecise arithmetic still satisfies the original conditions in Algorithms \ref{addm} and \ref{ga} if the specific rules are used in the code. For instance in (i) we define $c(v)$ such that $A_i c(v) = 1 + \frac{\varepsilon}{2}$ holds but in Algorithms \ref{addm} and \ref{ga} it is only required that $1 \leq A_i c(v) \leq  1 + \varepsilon$. Likewise, in (ii)(a), we have $A_i c(v) < 1 +\frac{\varepsilon}{4}$ but only $A_i c(v) < 1 +\frac{\varepsilon}{2}$ is required. We have similar situations in (ii) (b) and (c). 

Let us extend algorithms Algorithms \ref{addm} and \ref{ga} by a second coordinate function. Assume the first coordinate function $c:V \to \R^d$ is computed using imprecise arithmetic and the second coordinate function $\bar c:V \to \R^d$ is computed by exact arithmetic. Assume further that only $c$ is used to partition the set $V$. The above specifications are used in both cases.

By comparing the specifications with the requirements in the algorithms, we see that the correctness results still hold for imprecise arithmetic (using the specifications in the code) if the following condition holds for all $v \in V$ that occur in Algorithm \ref{ga} (those which occur in Algorithm \ref{addm} but not in Algorithm \ref{ga} are not relevant for the correctness results):
\begin{equation}\label{eq_err}
	\max_{i \in [m]} |A_i c(v) - A_i \bar c(v)| \leq \frac{\eps}{4}. 
\end{equation}
By the  Cauchy-Schwartz inequality we see that 
\begin{equation}\label{eq_err1}
	 \max_{i \in [m]} \|A_i^T\| \cdot \|c(v)- \bar c(v)\| \leq  \frac{\eps}{4}
\end{equation}	 
is sufficient for \eqref{eq_err} to hold, where $\|.\|$ denotes the Euclidian norm.

For a special class of polytopes the condition to guaranty correctness of Algorithms \ref{addm} and \ref{ga} can even be simplified. Assume that a ball around the origin with radius $\delta > 0$ is contained in the given polytope $P$. Then we have $\|A_i^T\| \leq \frac{1}{\delta}$ for all $i$. Thus the condition 
\begin{equation}\label{eq_err2}
 \|c(v)- \bar c(v)\| \leq  \frac{\eps \cdot \delta}{4}
\end{equation}	 
implies \eqref{eq_err1} and hence \eqref{eq_err}. Let us summarize these results.
\begin{corollary} Consider the approximate vertex enumeration problem for a polytope $P$ containing a ball around the origin with radius $\delta > 0$. Let 
$$ E \colonequals \max \{\|c(v)- \bar c(v)\| \mid \text{$v$ occurs during a run of Algorithm \ref{ga}} \}$$
denote the maximum error for the coordinate function caused by using imprecise arithmetic in Algorithm \ref{ga}. Then Algorithms \ref{addm} and \ref{ga} are correct if 
$$ E \leq \frac{\eps \cdot \delta}{4}.$$
\end{corollary}

Of course, the error $E$ caused by imprecise computations is difficult to quantify in practice, where computations with exact arithmetic shall be avoided. Nevertheless the result can help to evaluate the reliability of computational results. The larger the approximation error $\varepsilon> 0$ is chosen and the larger the radius $\delta>0$ of a ball around the origin inside of $P$, the more reliable the computational results are.

\section{Numerical results} \label{sec_num}

We present in this section numerical results for two examples of dimension $3$. In particular, we use them to compare the shortcut algorithm (SCA) with the approximate double description method (ADDM). 

The shortcut algorithm was implemented using a {\em half-edge data structure}, also known as {\em doubly connected edge list}, see e.g. \cite{CompGeom}, in order to store the planar graph. Most of the implementation is straightforward. To implement the second inner loop of Algorithm \ref{ga}, we iterate over the edges of the graph and store edges with one endpoint in $V_0$ and the other endpoint in $V_+$ in a queue. Then we remove an edge from the queue and walk around the incident face in order to insert new edges inside this face. If we meet a member of the queue we remove it from the queue. We repeat this procedure until the queue is empty.

Both algorithms were implemented in Python. The computations were made on a desktop computer with 2,6 GHz CPU clock speed. The examples were generated by {\tt bensolve tools} \cite{bt, bt-paper}.

\begin{example}\label{ex1}
	 Consider the linear image $P = \{M h \mid h \in H\}$ of a hypercube $H=[-1,1]\times ...\times [-1,1]$ of dimension $5^3=125$, where the linear mapping is given by the uniquely defined matrix $M \in \{-2,-1,0,1,2\}^{3 \times 125}$ with pairwise different columns. Figure \ref{ave_fig_2a} right shows (a good approximation of) the polytope. 
\end{example}

\begin{figure}[hpt]
	\begin{center}
	\includegraphics[width =.22\textwidth]{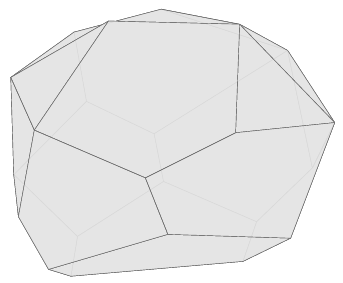}
	\includegraphics[width =.22\textwidth]{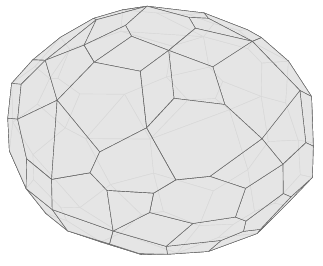}
	\includegraphics[width =.22\textwidth]{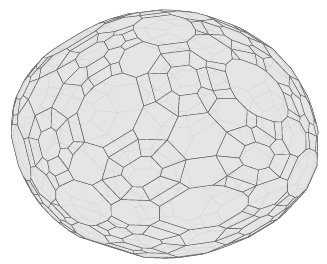}
	\includegraphics[width =.22\textwidth]{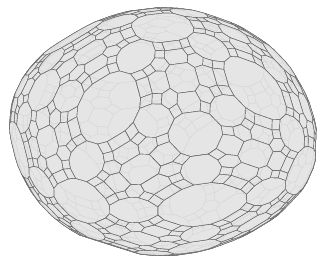}
	\end{center}
\caption{The polytope of Example \ref{ex1} computed by the shortcut algorithm for different tolerances $\varepsilon \in \{10^0,10^{-1},10^{-2},10^{-3}\}$ (from left to right). For tolerances smaller than $10^{-3}$, the pictures are ``almost identical'' to the right one.}\label{ave_fig_2a}
\end{figure}

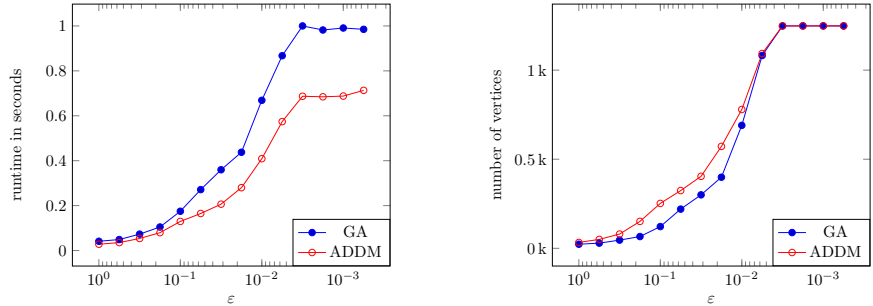
\begin{figure}[hpt]
\begin{center}
\begin{tikzpicture}[scale=.61]
\begin{axis}[xmode=log, x dir=reverse,
  xlabel=$\varepsilon$,legend style={at={(1,0)},anchor=south east},ylabel=runtime in seconds]
\addplot table [x=eps, y=time_ga, col sep=comma] {ave_fig_3.csv};
\addlegendentry{SCA}
\addplot [color=red,mark=o,red] table [x=eps, y=time_addm, col sep=comma] {ave_fig_3.csv};
\addlegendentry{ADDM}
\end{axis}
\end{tikzpicture}
\hspace{1cm}
\begin{tikzpicture}[scale=.61]
\begin{axis}[xmode=log, x dir=reverse,
  xlabel=$\varepsilon$,legend style={at={(1,0)},anchor=south east},ylabel=number of vertices, yticklabel = {
    \pgfmathparse{\tick/1000}
    \pgfmathprintnumber{\pgfmathresult}\,k
}]
\addplot table [x=eps, y=verts_ga, col sep=comma] {ave_fig_4.csv};
\addlegendentry{SCA}
\addplot [color=red,mark=o,red] table [x=eps, y=verts_addm, col sep=comma] {ave_fig_4.csv};
\addlegendentry{ADDM}
\end{axis}
\end{tikzpicture}
\end{center}
\caption{Left: Runtime of the shortcut algorithm (SCA) compared to the runtime of the approximate double description method (ADDM) for Example \ref{ex1} in dependence of the tolerances. Right: Number of vertices computed by the shortcut algorithm (SCA) compared to the number of vertices computed by the approximate double description method (ADDM) for Example \ref{ex1} in dependence of the tolerances. For $\varepsilon \in \{10^{-3},10^{-3.25},10^{-3.5},\dots,10^{-14}\}$ (not displayed here) the runtime of both algorithms is nearly constant and the number of vertices is exactly constant.}\label{ave_fig_3}
\end{figure}

In Figure \ref{ave_fig_2a} we see that larger tolerances $\varepsilon$ can leads to simpler approximations. Figure \ref{ave_fig_3} left shows a runtime comparison of both algorithms. For Example \ref{ex1}, the approximate double description method is slightly faster than the shortcut algorithm. We also see in Figure \ref{ave_fig_3} left that the coarser the approximation is the less computational time is required by both algorithms. An advantage of the shortcut algorithm can be seen in Figure \ref{ave_fig_3} right. For larger tolerances it computes strictly less vertices than the approximate double description method (recall that the vertices computed by the shortcut algorithm are always a subset of the vertices computed by the approximate double description method).

\begin{figure}[hpt]
	\begin{center}
		\includegraphics[width =.22\textwidth]{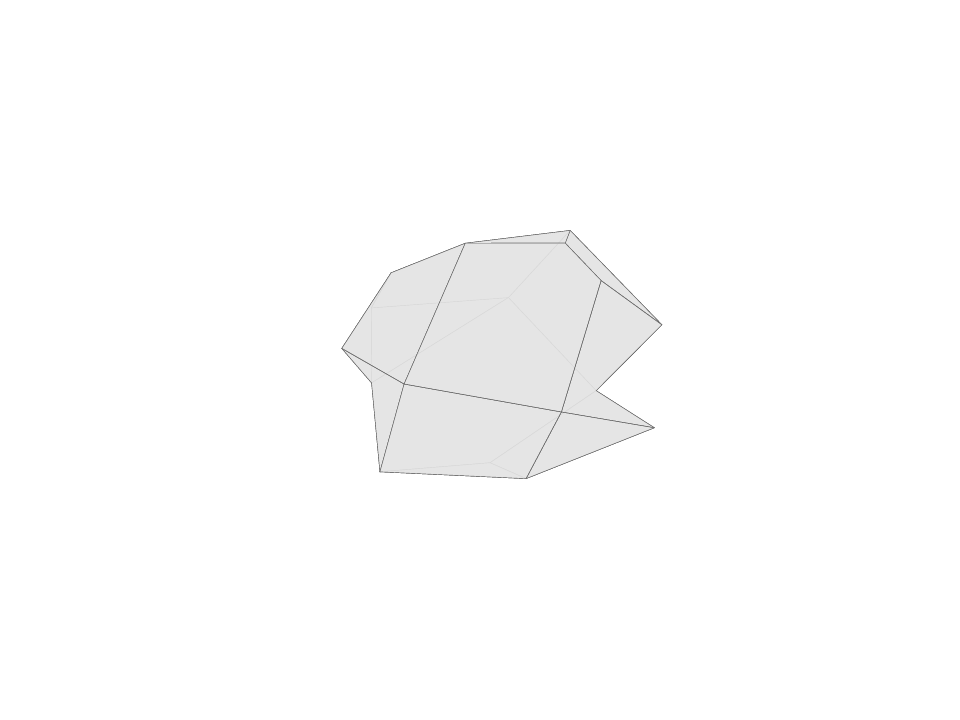}
		\hspace{.5cm}
		\includegraphics[width =.22\textwidth]{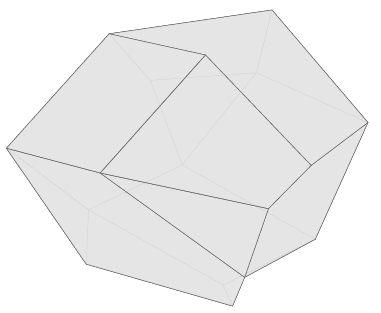}
		\hspace{.5cm}
		\includegraphics[width =.22\textwidth]{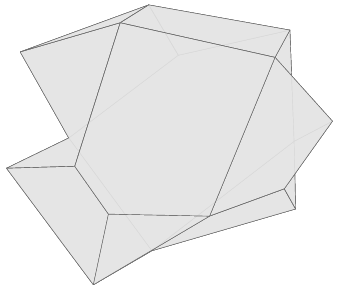}
	\end{center}
\caption{The polytope of Example \ref{ex1} computed by the shortcut algorithm with a huge tolerance $\varepsilon=2$ from different viewpoints in order to visualize the non-convex nature of the construction.}\label{ave_fig_2b}
\end{figure}

Figure \ref{ave_fig_2b} visualizes the fact that faces computed by the shortcut algorithm are bounded by skew polygons, i.e.\ they do not necessarily belong to a plane. In general, the algorithm does not produce (convex) polytopes, not even after a triangulation of the faces. Nevertheless the vertices computed by the shortcut algorithm provide an approximate V-representation. Moreover, the visualization of the non-convex objects makes sense for practical reasons as they give us an impression of both the approximation and the original polytope. For small tolerances the resulting objects of the shortcut algorithm are ``close to'' the given convex polytopes.

Let us turn to the second example, which provides a sequence of polytopes with increasing complexity. The polar of a polytope $P$ is defined as $P^\circ \colonequals \{y \in \R^d\mid \forall x \in P:y^T x \leq 1\}$.

\begin{example}\label{ex2}
	A sequence of polytopes $P_i$ is defined recursively. $P_0$ is a regular simplex in $\R^3$ with edge length $1$ symmetrically placed around the origin. $P_i$ is defined as the Minkowski sum of $P_{i-1}$ and the polar of $P_{i-1}$. Figure \ref{ave_fig_7} shows some of the polytopes and Figure \ref{ave_fig_1} shows some  approximations computed by the shortcut algorithm. 
\end{example}

\begin{figure}[hpt]
	\begin{center}
	\includegraphics[width =.22\textwidth]{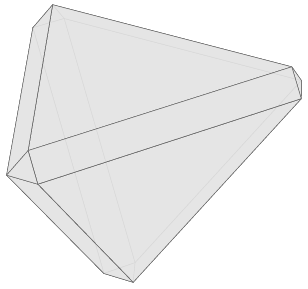}
	\includegraphics[width =.22\textwidth]{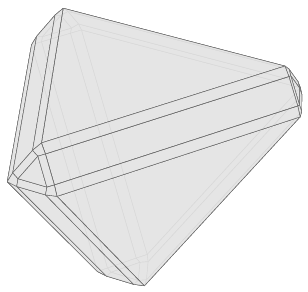}
	\includegraphics[width =.22\textwidth]{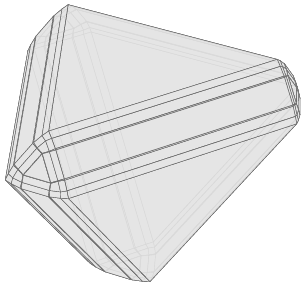}
	\includegraphics[width =.22\textwidth]{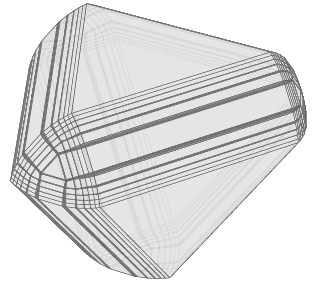}
	\end{center}
\caption{The polytopes $P_1$, $P_2$, $P_3$ and $P_6$ (from left to right) from Example \ref{ex2}. Approximations of $P_6$ can be seen in Figure \ref{ave_fig_1}}\label{ave_fig_7}
\end{figure}

\begin{figure}[hpt]
	\begin{center}
	\includegraphics[width =.22\textwidth]{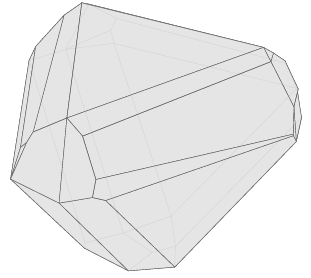}
	\includegraphics[width =.22\textwidth]{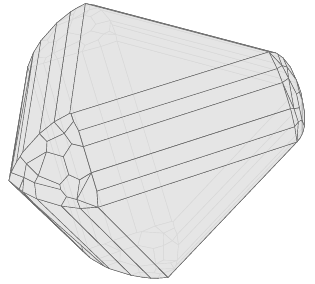}
	\includegraphics[width =.22\textwidth]{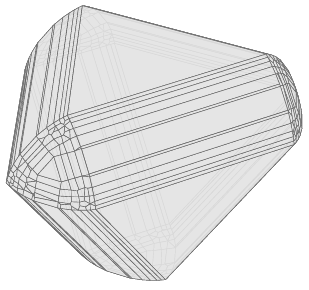}
	\includegraphics[width =.22\textwidth]{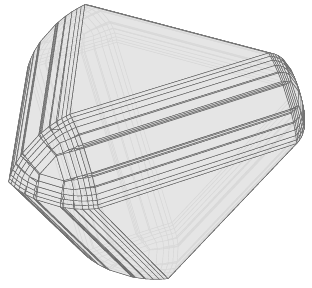}
	\end{center}
\caption{Approximations of $P_6$ from Example \ref{ex2} computed by the shortcut algorithm. The tolerance was chosen as $\varepsilon \in \{10^{-1},10^{-2},10^{-3},10^{-4}\}$ (from left to right).}\label{ave_fig_1}
\end{figure}

In Figure \ref{ave_fig_6} left we see that for Example\ref{ex2} the approximate double description method is slower than the shortcut algorithm for certain tolerances. Also, the approximate double description method computes way too many vertices for some tolerances, see Figure \ref{ave_fig_6} center. While the shortcut algorithm works well for arbitrary tolerances, the approximate double description method produces satisfactory results only for sufficiently small tolerances.

Finally, in Figure \ref{ave_fig_6} right we observe the runtime of both algorithms in case of increasing complexity. Since $\varepsilon >0$ was chosen small enough, the approximate double description method performs slightly better than the shortcut algorithm. 

A further advantage of the shortcut algorithm is that the results can easily be visualized by using the planar graph stored as a half-edge data structure and the coordinate function while further computational steps are necessary for a similar task for the approximate double description method.

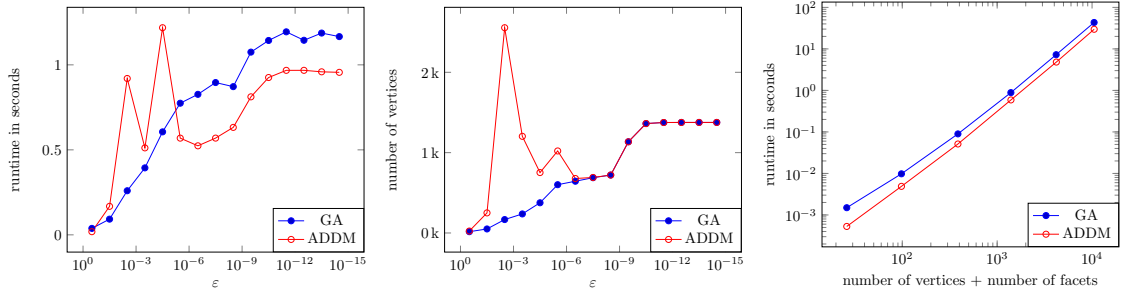
\begin{figure}[hpt]
	\begin{center}
\begin{tikzpicture}[scale=.57]
\begin{axis}[xmode=log, x dir=reverse,
  xlabel=$\varepsilon$,legend style={at={(1,0)},anchor=south east},ylabel=runtime in seconds]
\addplot table [x=eps, y=time_ga, col sep=comma] {ave_fig_6a.csv};
\addlegendentry{SCA}
\addplot [color=red,mark=o,red] table [x=eps, y=time_addm, col sep=comma] {ave_fig_6b.csv};
\addlegendentry{ADDM}
\end{axis}
\end{tikzpicture}
\begin{tikzpicture}[scale=.57]
\begin{axis}[xmode=log, x dir=reverse,
  xlabel=$\varepsilon$,legend style={at={(1,0)},anchor=south east},ylabel=number of vertices, yticklabel = {
    \pgfmathparse{\tick/1000}
    \pgfmathprintnumber{\pgfmathresult}\,k
}]
\addplot table [x=eps, y=verts_ga, col sep=comma] {ave_fig_6a.csv};
\addlegendentry{SCA}
\addplot [color=red,mark=o,red] table [x=eps, y=verts_addm, col sep=comma] {ave_fig_6b.csv};
\addlegendentry{ADDM}
\end{axis}
\end{tikzpicture}
\begin{tikzpicture}[scale=.57]
\begin{axis}[xmode=log,ymode=log,
  xlabel=number of vertices + number of facets,legend style={at={(1,0)},anchor=south east},ylabel=runtime in seconds]
\addplot table [x=verts_ga, y=time_ga, col sep=comma] {ave_fig_5.csv};
\addlegendentry{SCA}
\addplot [color=red,mark=o,red] table [x=verts_addm, y=time_addm, col sep=comma] {ave_fig_5.csv};
\addlegendentry{ADDM}
\end{axis}
\end{tikzpicture}
\end{center}
\caption{Computational results for Example \ref{ex2}. Left: Runtime comparison for the polytope $P_4$ with various tolerances. The shortcut algorithm (SCA) performs better for some (larger) tolerances. The approximate double description method is very slow for certain tolerances. For instance, for $\varepsilon=10^{-4}$, which is not displayed here as we have chosen $\varepsilon \in \{10^{-0.5},10^{-1.5},\dots,10^{-14.5}\}$, it took $5013$ seconds and the result had $810$ vertices. The reason for this long runtime is that extremely many vertices were computed in intermediate steps. The approximate double description method worked well for sufficiently small tolerances. Center: Comparison of the number of vertices computed by both algorithms for the polytope $P_4$ with various tolerances. The approximate double description method computes way too many vertices for some (larger) tolerances. For instance, for $\varepsilon=10^{-3}$, which is not displayed here, the result had $9456$ vertices and was computed in $3.4$ seconds. Right: The runtime of both algorithms for $\varepsilon = 10^{-9}$ is shown in dependence of the ``size'' (number of vertices + number of faces) of the polytopes $P_1,P_2,\dots,P_6$ (from left to right).}\label{ave_fig_6}
\end{figure}

\section{Conclusions, open questions and comments}

The approximate vertex enumeration problem was introduced and motivated. Two solution methods for dimension $d \in \{2,3\}$ were introduced, were shown to be correct and tested by numerical examples, the approximate double description method and the shortcut algorithm. Both methods remain correct when imprecise arithmetic is used and the computational error is sufficiently small.

While the approximate double description method was formulated for arbitrary dimension, the shortcut algorithm makes sense only for dimension $d \in \{2,3\}$, because planarity of the constructed graph is utilized.

The approximate double description method was shown to be correct also for dimension $d \geq 4$ if an  impracticable assumption is satisfied in each iteration. The following questions remain open:
\begin{enumerate}[(1)]
	\item  What is the smallest dimension $d$ (if there is any) such that the approximate double description method fails (without any additional assumption)? 
	\item Is there any (other) ``practically relevant'' (compare Remark \ref{rem_trivial}) solution method for the approximate vertex enumeration problem for dimension $d\geq 4$?
	\item  How reliable are vertex enumeration methods using imprecise arithmetic for dimension $d\geq 4$?
\end{enumerate}

\bigskip\noindent {\bf Acknowledgements.} The author thanks Michael Joswig for inspiring to this research and Benjamin Wei{\ss}ing and David Hartel for the interesting discussions on the subject.


\end{document}